\documentclass[11pt]{amsart}
\usepackage{amssymb, amsmath, amsthm}

\usepackage{color}
\usepackage{graphicx}
\usepackage{aurical}
\usepackage[final]{hyperref}
\usepackage[a4paper, centering]{geometry}
\usepackage{pdfsync}

\usepackage[mathscr]{eucal}
\usepackage{lmodern}

\usepackage{pgfplots}
\usepackage{subfig}
\pgfplotsset{compat=1.10}

\usetikzlibrary{calc,intersections,through,backgrounds}
\usetikzlibrary{quotes,angles,shapes.geometric,shapes.misc,shapes.arrows}
\usetikzlibrary{hobby,graphs,arrows,spy,patterns,shadings,fadings,shadows}
\usepgfplotslibrary{fillbetween}
\tikzset{>=latex}
\usepackage{tikz-3dplot}

\newif\ifextended
\extendedfalse

\newtheorem{teor}{Theorem}[section]
\newtheorem{prop}[teor]{Proposition}

\theoremstyle{definition}
\newtheorem{Definition}[teor]{Definition}
\theoremstyle{remark}
\newtheorem{rk}[teor]{Remark}

\long\def\elimina#1{}

\def\R{\mathbb{R}}

\def\Z{\mathbb{Z}}

\def\haus{\mathcal{H}^{1}}
\def\leb{\mathcal{L}^{2}}

\def\Lip{\textrm{Lip}}

\def\ne{\nu^E}
\def\curf{\kappa}
\def\eps{\varepsilon}
\def\acube{\text{\Fontskrivan A}_\eps}
\def\bcube{\text{\Fontskrivan B}_\eps}

\def\disc{\Xi}

\def\dist{d}

\def\cray#1{{[\![#1]\!]}}

\DeclareMathOperator{\inte}{int} 
 \DeclareMathOperator{\dive}{div}

\title[Crystalline evolutions in chessboard--like microstructures]%
{Crystalline evolutions in chessboard--like microstructures}%
\author[A.~Malusa]{Annalisa Malusa$^\dagger$}
\thanks{$^\dagger$ Dipartimento di Matematica ``G.\ Castelnuovo'', Sapienza 
Università di Roma, Piazzale Aldo  Moro 2, 00185 Roma, Italy, email: 
malusa@mat.uniroma1.it} 
\author[M.~Novaga]{Matteo Novaga$^\ddagger$}
\thanks{$^\ddagger$ Dipartimento di Matematica, Università di Pisa, Largo 
B. Pontecorvo 5, 56217 Pisa, Italy, email: matteo.novaga@unipi.it}

\date{\today}
\begin{document}

\begin{abstract}
We describe the macroscopic behavior of evolutions by crystalline curvature of 
planar sets in a 
chessboard--like medium,
modeled by  a periodic forcing term. We show that the underlying  
microstructure may produce both pinning and confinement effects on the 
geometric 
motion.
\end{abstract}

\maketitle

\tableofcontents

\section{Introduction}

We are concerned with the asymptotic behavior of motions of planar curves 
according 
to the law
\begin{equation}\label{eqevol}
v = \curf + g\left(\frac{x}{\varepsilon},\frac{y}{\varepsilon}\right),
\end{equation}
where $v$ is the normal velocity, $\curf$ is the crystalline curvature, 
$g:\R^2\to \R$ is a periodic forcing term, with average $\bar g$,
and $\varepsilon>0$ is a small parameter which takes account of the frequency 
of oscillation. 

Crystalline evolutions provide simplified models for describing several 
phenomena in Materials Science
(see \cite{Gu,Ta, TCH} and references therein) and have been significantly 
studied in recent years
(see for instance \cite{AT, GG, BNP1,BNP2, CMP, CMNP, CN, Gi}). 

The forcing term $g$
models a rapidly oscillating heterogeneous medium and, in the 
homogenization limit $\varepsilon \to 0$,
the oscillations of the medium affect the velocity of the evolving front.
The geometric motion
\eqref{eqevol} corresponds to the gradient flow of the energy 
\[
F_\varepsilon(E)=\int_{\partial E} 
\bigl( |\nu_1^E|+|\nu_2^E|\bigr) \, d\haus +\int_E 
g\Bigl(\frac{x}{\varepsilon},\frac{y}{\varepsilon}\Bigr)\ 
d\leb,\qquad E\subset\R^2,
\]
where we identify the evolving curve with the boundary of a set $E$.
Since the volume term converges to 
$\bar g\,\leb(E)$ as $\varepsilon \to 0$, the $\Gamma$-limit of the 
functionals $F_\varepsilon$ (see \cite{GCB}) is given by
\[
\overline{F}(E)=\int_{\partial E} 
\bigl( |\nu_1^E|+|\nu_2^E|\bigr)\, d\haus +\bar g\,\leb(E).
\]
Hence, our analysis can be set in a large class of variational evolution 
problems dealing with limits of motions driven by functionals $F_\varepsilon$ 
depending on a small parameter. 

For oscillating functionals like $F_\varepsilon$, the energy landscape of the energies 
can be quite different from that of their $\Gamma$-limit, and the related
motions can be influenced by the presence of local minima which may give rise
to pinning phenomena, or
to effective homogenized velocities (see \cite{NV,BCN,BGN,Br}). 
In the case of geometric motions, a 
general understanding of the effects of microstructure is still missing.
Recently, some results have been obtained for two-dimensional crystalline
energies, for which a simpler description can be given in terms of a system
of ODEs (see for instance \cite{BGN,BSc,BCY,BSo,BGN,BMN}). 

\smallskip

Coming back to our specific problem \eqref{eqevol}, 
we assume for simplicity that $g$
takes only two values $\alpha<0<\beta$, its periodicity cell is $[0,1]^2$,
and that $\bar g=\frac{\alpha+\beta}{2}$. 
We also assume that $g$ has a 
chessboard structure, as specified in \eqref{eqg} below.


After a careful analysis, it turns out that curves evolving by \eqref{eqevol}
undergo a microscopic 
``facet-breaking'' phenomenon at a scale $\varepsilon$, with small segments of length proportional 
to $\varepsilon$ being created and, in some cases, subsequently reabsorbed. 
The macroscopic effect of this behavior is a ``pinning effect'' for the limit evolution,
corresponding to the possible onset 
of new edges, with slope of 45 degrees and zero velocity
(depending on the initial set and on the values $\alpha,\,\beta$).
On the other hand, the horizontal and the vertical edges always travel with the asymptotic velocity 
$\kappa + \bar g$. We thus obtain a characterization of the limit evolution, which is the
main result of this paper, and is stated precisely in Theorems \ref{d:square} and \ref{d:rect}.

We point out that, due to the possible presence of these new edges with zero velocity,
the limit flow does not coincide with the gradient flow of the limit functional
$\overline{F}$, which is simply given by $v=\curf + \bar g$. 

We recall that, in a previous paper \cite{BMN}, we considered a similar homogenization problem
where the periodic function g depends only on the horizontal variable, so that the medium 
has a stratified, opposite to a chessboard-like, structure.

It would be very interesting to extend our analysis to the isotropic variant of \eqref{eqevol},
where the crystalline curvature $\kappa$ is replaced by the usual curvature of the evolving curve, 
so that  \eqref{eqevol} becomes a forced curvature flow. However, as such evolution cannot be described in terms of a system of ODEs, different techniques would be needed (partial results in this direction can be found in \cite{CNV,CDN}).

\smallskip

The plan of the paper is the following: in Section \ref{s:sett} we introduce 
the notion of crystalline curvature and the evolution problem we are interested 
in. In Section \ref{s:calibr} we introduce the notion of calibrable edge, that 
is, an edge which does not break during the evolution, and we state the 
calibrability conditions. Finally, in Section 
\ref{s:motion} we characterize explicitly the limit evolution 
as $\varepsilon \to 0$, for initial data which are squares (Theorem \ref{d:square}) 
and more generally rectangles (Theorem \ref{d:rect}).

\smallskip 

\subsection*{Acknowledgments.} The authors wish to thank Andrea Braides for useful discussions on the topic of this paper.
The second author was partially supported by the Italian CNR-GNAMPA and by the University of Pisa via grant PRA-2017 ``Problemi di ottimizzazione e di evoluzione in
ambito variazionale''.

\section{Setting of the problem}\label{s:sett}

\subsection*{Notation}
 
The canonical basis of $\R^2$ will be denoted by $e_1=(1,0)$, $e_2=(0,1)$.
 
The 1--dimensional 
Hausdorff measure and the 2--dimensional Lebesgue measure in $\R^2$ will be 
denoted by $\haus$ and $\leb$, respectively.

We say that a set $E\subseteq \R^2$ is a  {\em  Lipschitz set} if its boundary
$\partial E$  can be written, locally, as the graph of a Lipschitz function 
(with respect to a suitable orthogonal coordinate system). 
The  {\em  outward normal} to $\partial E$ at $\xi$, that exists 
$\haus$--almost everywhere on 
$\partial E$, will be denoted by  $\ne=(\nu_1^E,\nu_2^E)$.

The Hausdorff distance between two sets
$E,\ F\in \R^2$ will be denoted by $d_H(E,F)$.

\subsection*{The crystalline curvature}

We briefly recall a notion of curvature $\kappa^E$ on 
$\partial E$
which is 
consistent with the requirement that a geometric evolution $E(t)$, 
reducing as fast as possible the energy
\[
P_\varphi(E):= \int_{\partial E} 
\bigl(|\nu_1^E|+|\nu_2^E|\bigr)\, d\haus,
\]
has normal velocity $\kappa^{E(t)}$ 
$\haus$--almost everywhere on $\partial E(t)$. 

The surface tension
$\varphi^\circ(x,y)=|x|+|y|$ is the polar function of the convex  norm 
$\varphi(x,y)=\max\{|x|,|y|\}$, $(x,y)\in\R^2$, so that
$P_\varphi(E)$ turns out to be the perimeter associated to the anisotropy $\varphi(x,y)$,
that is, the Minkowski content 
obtained by considering $(\R^2,\varphi)$ as a normed space.
The sets $\{\phi(\xi)\leq 1\}$ and 
$\{\phi^\circ(\xi)\leq 1\}$ are the 
square 
$K=[-1,1]^2$, and the square with corners at $(\pm 1,0)$ and $(0,\pm 1)$, 
respectively.

Given a nonempty compact set $E\subseteq \R^2$, if we denote by $\dist^E$ the 
{\em oriented  $\varphi$--distance function} to $\partial E$, negative inside 
$E$, that is,
\[
\dist^E(\xi):=\inf_{\eta\in E}\varphi(\xi-\eta) - \inf_{\eta\not\in  
E}\varphi(\xi-\eta), 
\qquad \xi\in\R^2.
\]
The normal cone at $\xi\in\partial E$ is well defined whenever $\xi$ is 
a differentiability point 
for $\dist^E$, and it is given by 
$T_{\phi^\circ}\bigl(\nabla\dist^E(\xi)\bigr)$,
where 
\[
T_{\varphi^\circ}(\xi^\circ) :=\{\xi\in \R^2, 
\xi\cdot\xi^\circ=(\varphi^\circ(\xi))^2\}, \quad \xi^\circ\in\R^2\,.
\]
The notion of intrinsic curvature in $(\R^2, \varphi)$ is based on the 
existence of regular selections of 
$T_{\phi^\circ}\bigl(\nabla\dist^E\bigr)$ on $\partial E$.

\begin{Definition}[$\varphi$--regular set, Cahn--Hoffmann field,  
$\varphi$--curvature] 
We say that a set $E\subseteq \R^2$ is {\em  $\varphi$--regular} if 
$\partial E$ is 
a compact Lipschitz curve, and there exists a 
vector field $n_\varphi\in \Lip(\partial E;\R^2)$ such that 
$n_\varphi \in T_{\varphi^\circ}(\nabla \dist^E)$ $\haus$--almost everywhere in 
$\partial E$.

\noindent Any selection of the multivalued 
function $T_{\varphi^\circ}(\nabla \dist^E)$
on $\partial E$ is called  a {\em  Cahn--Hoffmann vector field} for $\partial 
E$, and $\curf=\dive n_\varphi$ is the related {\em  
$\varphi$--curvature} (or {\em crystalline curvature}) of $\partial E$.
\end{Definition}




\begin{rk}[Edges and vertices]\label{r:nvertex}
A direct computation gives that $T_{\phi^\circ}(\xi^\circ)$ is a singleton
if $\varphi^\circ(\xi^\circ)=1$, and $\xi^\circ$ is not a coordinate vector.
Moreover one gets
\[ 
\begin{cases}
T_{\varphi^\circ}(e_1) & =\cray{(1,1),(1,-1)}, \\
T_{\varphi^\circ}(e_2) & =\cray{(-1,1),(1,1)},\\ 
T_{\varphi^\circ}(-e_1) & =\cray{(-1,1),(-1,-1)},\\ 
T_{\varphi^\circ}(-e_2) & =\cray{(-1,-1),(1,-1)}.
\end{cases}
\]
(Here and in the following $\cray{\xi,\eta}$
is the closed segment joining the vector $\xi$ with $\eta$).
The boundary of a $\varphi$--regular set $E$ is given by a finite number
of maximal closed arcs with the property that $T_{\varphi^\circ}(\nabla 
\dist^E)$
is a fixed set $T_A$ in the interior of each arc $A$. This set 
$T_A$ 
is either a singleton,  if the arc $A$ is not a horizontal or vertical 
segment, or one of the closed convex cones described above. 
The maximal arcs of $\partial E$ which are straight horizontal or 
vertical segments 
will be called {\em  edges}, and the endpoints of every arc will be called 
{\em  vertices} of $\partial E$.

The requirement of Lipschitz continuity keeps the  value of every 
Cahn--Hoffmann vector field fixed at vertices.
Hence, in 
order to exhibit a Cahn--Hoffmann vector field $n_\varphi$ on $\partial E$ it 
is enough to construct a field $n_A\in \Lip(A;\R^2)$ on each arc $A$, with the 
correct values 
at the vertices, and satisfying the constraint $n_A 
\in T_A$.
In what follows, with a little abuse of notation, we shall call $n_A$ the 
Cahn--Hoffmann vector field on the arc $A$.
\end{rk}

\subsection*{Forced crystalline flows}
Let $\alpha < 0 < \beta$, and let  $g\colon \R^2 \to \R$ be the function
defined in $[0,1]^2$ by 
\begin{equation}\label{eqg}
g(x,y)=
\begin{cases}
\alpha, & \text{in}\ \left]0,\dfrac{1}{2}\right[^2 \bigcup 
\left]\dfrac{1}{2},1\right[^2, \\[8pt]
\beta, & \text{in}\ \left(\left]\dfrac{1}{2},1\right[ \times 
\left]0,\dfrac{1}{2}\right[ \right) \bigcup \left(\left]0,\dfrac{1}{2}\right[ 
\times 
\left]\dfrac{1}{2},1\right[ \right),
\end{cases}
\end{equation}
and extended by periodicity in $\R^2$.
For $\eps>0$, let $g_{\epsilon}(x,y)=g(\frac{x}{\eps},\frac{y}{\eps})$.

We will denote by $\acube$ (resp. $\bcube$) the union of all closed squares 
$Q$ of side length $\eps$ such that $g_\eps=\alpha$ (resp. $g_\eps=\beta$) in 
the interior of $Q$.
The set of discontinuity points of $g_\eps$ will be denoted by $\disc$. A 
{\em discontinuity line} is a straight line contained in $\disc$. 

We define the 
multifunction $G_{\epsilon}$ in $\R^2$, by setting  
$G_{\epsilon}=[\alpha,\beta]$ on $\disc$, and 
$G_{\epsilon}(x,y)=\{g_\eps(x,y)\}$ in
$\R^2\setminus \disc$.

We want to introduce our notion of geometric evolution $E(t)$, obeying to the 
law
\begin{equation}\label{f:evlow}
V=\kappa+g_\varepsilon, \qquad \text{on}\ \partial E,
\end{equation}
where $V$ is the normal velocity, and $\kappa$ is the 
crystalline curvature on $\partial E(t)$.

In order to make a sense to \eqref{f:evlow} it would be enough to require that
the evolution is a family of $\varphi$--regular sets. Nevertheless, as 
underlined in Remark \ref{r:nvertex}, even if $E$ is a $\varphi$--regular set, 
the crystalline curvature on $\partial E$ may not be uniquely determined, 
due to the infinitely many choices  for the Cahn--Hoffmann vector field on the 
edges of $\partial E$. 

This ambiguity can be overcome introducing an additional postulate, which is 
consistent with the notion of forced curve shortening flow (see 
\cite{BNP1}, \cite{BNP2}, \cite{BNP3}, \cite{GR1}, \cite{GR2}).

\begin{Definition}[Variational Cahn--Hoffmann field] \label{d:vch}
A {\em  variational Cahn--Hoffmann vector field} for a $\varphi$--regular set
$E$ is a Cahn--Hoffmann vector field $n$ on $\partial E$ such that for every  
edge
$L$ of $\partial E$ not lying on a discontinuity line of $g_\eps$, the 
restriction $n_{L}$ of $n$ on $L$ is 
the unique minimum of the functional
\[
\mathcal{N}_{L}(n)= \int_{L} |g_\eps-\dive n|^2\, d\haus
\] 
in the set
\[
D_L=\left\{ 
n \in L^\infty(L,\R^2), n\in T_{L}, \dive n\in L^\infty(L),
n(p)=n_0, n(q)=n_1
\right\}
\]
where $p$, $q$ are the endpoints of $L$ and $n_0$, $n_1$ are the 
values at $p$, $q$ assigned to every Cahn--Hoffmann vector field
(see Remark \ref{r:nvertex}). 
\end{Definition}

\begin{rk} 
If the minimum $n_{L}$ in $D_L$ of the functional $\mathcal{N}_{L}$ 
satisfies the strict constraint $n_{L}(\xi)\in \inte T_{L}$ for every 
$\xi\in L$, then the velocity $g_\eps-\dive n_{L}$ is constant along the edge, 
that is the flat arc remains flat under the evolution. This is always the case
for unforced crystalline flows, since the unique minimum is the
interpolation of the assigned values at the vertices of $L$, and  the constant 
value 
of the $\varphi$--curvature is given by
\[ 
\kappa^{L}=\chi_{L}\frac{2}{\ell} \ \text{on}\ L,
\]
where $\ell$ is the length of the edge $L$ and $\chi_{L}$ is a 
convexity factor:
$\chi_{L}=1,-1,0$, depending on whether $E(t)$ is locally convex at $L$, 
locally 
concave at $L$, or neither.
\end{rk}

\begin{Definition}[Forced crystalline evolution] \label{d:flow}
Given $T>0$, we say that a family $E(t)$, $t\in [0,T]$, is a  
{\em forced crystalline  
curvature flow} (or {\em forced crystalline evolution}) 
in $[0,T)$
if 
\begin{itemize}
\item[(i)] $E(t)\subseteq\R^2$ is a Lipschitz set for every $t\in [0,T)$;
\item[(ii)] there exists an open set $A\subseteq \R^2\times [0,T)$ such that
$\bigcup_{t\in [0,T)}\partial E(t)\times \{t\}\subseteq A$, and the function
$d(\xi,t)\doteq d^{E(t)}(\xi)$ is locally Lipschitz in $A$;
\item[(iii)] there exists a function $n \in L^\infty(A,\R^2)$, with  
$\dive n\in L^\infty(A)$, such that the restriction of  $n(t, \cdot)$ to
$\partial E(t)$ is a variational Cahn--Hoffmann vector field for $\partial 
E(t)$ for almost every $t\in [0,T]$;
\item[(iv)] $\partial_t d-\dive n \in G_\eps$ 
$\haus$--almost everywhere in $\partial E(t)$ and for almost every $t\in[0,T)$.
\end{itemize}
\end{Definition}

\section{Calibrability conditions}\label{s:calibr}

In this section we deal with the minimum problem in Definition
\ref{d:vch} for a given $\varphi$--regular set $E$, and we characterize the 
edges of $\partial E$ having constant velocity $v_L:=\kappa^L+g_\eps$.

The results concern edges $L\in \partial E$ not lying on a discontinuity line 
of the forcing term, in such a way that $g_\eps$ is defined $\haus$--almost 
everywhere
on $L$. 
We will use the notation  $L=[p,q]\times \{y\}$ or $L=\{x\} \times [p,q]$, with 
$x,y \not\in \frac{\varepsilon}{2}\Z$, so that  $\ell=q-p$. 

Setting by $n\colon \R \to [-1,1]$ the unique 
varying component of the variational Cahn--Hoffmann 
vector field on $L$ (recall Remark \ref{r:nvertex}), the assigned 
values of $n$ are the following:
\begin{equation}\label{e:bcon}
(BV)=\begin{cases}
n(p)=n(q)=n_0\in \{\pm 1\} &  \text{if}\ \chi_L=0; \\
n(p)=-1,\ n(q)=1, &  \text{if}\ \chi_L=1; \\
n(p)=1,\ n(q)=-1, &  \text{if}\ \chi_L=-1.
\end{cases}
\end{equation}

Moreover, we denote by
$\gamma_\eps\colon \R \to \R$ the restriction of $g_\eps$ on the straight line
containing $L$, and we distinguish two different type of discontinuity points 
for $\gamma_\eps$:
\[
\mathcal{I}_{\beta,\alpha}=\{s\in \R \colon\ \gamma_\eps=\alpha\ \text{in} \ 
(s,s+\eps/2)\},
\quad
\mathcal{I}_{\alpha,\beta}=\{s\in \R \colon\ \gamma_\eps=\beta\ \text{in} \ 
(s,s+\eps/2)\}.
\]

With these notation, the requirement that
$\kappa+g_\eps$ is constant on $L$  can be rephrased in the following 1D 
problem.

\begin{Definition}[Calibrability conditions]\label{d:cali}
$L$ is a {\em calibrable edge} of $\partial E$ if and only 
if there exists a Lipschitz function $n\colon [p,q]\to \R$
such that the following hold.
\begin{itemize}
\setlength\itemsep{5pt}
\item[(i)] $n$ satisfies \eqref{e:bcon}.
\item[(ii)] $|n|\leq 1$ in $[p,q]$. 
\item[(iii)] $\displaystyle n'+\gamma_\varepsilon= \chi_L\frac{2}{\ell}+ 
\frac{1}{\ell} \int_p^q 
\gamma_\eps(s) \, ds$ a.e.\ in $[p,q]$.
\end{itemize}

In this case, we say that $v_L=n'+\gamma_\eps$ is the (normal) velocity of the 
edge 
$L$.
\end{Definition}

The calibrability property was studied, in its full generality, in \cite{BMN}. 
We collect here the results needed in the rest of the paper, sketching
the proofs for sake of completeness.

Denoting by $\ell_{\alpha}$, $\ell_{\beta}\in [0,\varepsilon/2]$ the 
non--negative lengths given by the conditions 
\begin{equation}\label{e:pezzetti}
 \ell - 
 \varepsilon\left\lfloor\dfrac{\ell}{\varepsilon}\right\rfloor=\ell_{\alpha}+ 
 \ell_{\beta}, \qquad 
 \int_L \gamma_\varepsilon(s)\, ds= \frac{\alpha+\beta}{2}\left(\ell-
 \ell_{\alpha}-\ell_{\beta}\right)+\alpha \ell_{\alpha}+\beta \ell_{\beta},
\end{equation}
the calibrability condition in Definition \ref{d:cali}(iii) sets the value of 
$n'$ outside 
the jump set of $\gamma_\varepsilon$:
\begin{equation}\label{f:conden3}
n'(s) = 
 \begin{cases}
\dfrac{1}{2\ell}\left(4 
\chi_L+(\beta-\alpha)(\ell-\ell_{\alpha}+\ell_{\beta})\right)
&  \text{if}\ 
 \gamma_\varepsilon(s)=\alpha , \\[10pt]
 \dfrac{1}{2\ell}\left(4 
 \chi_L-(\beta-\alpha)(\ell+\ell_{\alpha}-\ell_{\beta})\right),
  & \text{if}\ 
 \gamma_\varepsilon(s)=\beta.
 \end{cases}
\end{equation}
so that $n$ needs to be
\begin{equation}\label{f:CHespl}
n(s)=n(p)+(s-p)v_L-\int_p^s \gamma_\varepsilon(\tau)\, d\tau, 
\end{equation}
where $v_L$ is the feasible velocity of the edge $L$
\begin{equation} \label{f:velocita}
v_L=\chi_L\frac{2}{\ell}+\frac{\alpha+\beta}{2}+
\frac{\beta-\alpha}{2\ell}(\ell_\beta -\ell_\alpha).
\end{equation} 

In conclusion, the calibrability conditions (i) and (iii) in Definition 
\ref{d:cali} fix univocally a
candidate field \eqref{f:CHespl} which is continuous 
and affine
with given slope in each phase of $\gamma_\varepsilon$. This field $n$
is the Cahn--Hoffman field which calibrates $L$ with velocity 
\eqref{f:velocita} if and only if it also satisfies the constraint
$|n(x)| \leq 1$ for every $x\in [p,q]$.

\begin{rk} \label{f:asseps}
In what follows we will assume 
$0<\varepsilon < \dfrac{8}{\beta-\alpha}$
in such a way that the small perturbation
$\chi_L\dfrac{2}{\ell}+
\dfrac{\beta-\alpha}{2\ell}(\ell_\beta -\ell_\alpha)$ has the same sign of the 
curvature term $\chi_L\dfrac{2}{\ell}$. 
\end{rk}

%
%
 
\begin{prop}\label{r:hcurvz}
Let $L$ be an edge  with zero 
$\varphi$--curvature, and  let 
$n_0\in \{\pm 1\}$ be the given value of the 
Cahn--Hoffmann vector field at the endpoints of $L$.
Then the following 
hold.
\begin{itemize}
\item[(i)] If $\ell=\ell_\alpha+\ell_\beta < \varepsilon$, 
$L$ is calibrable with velocity 
$
v_L=\dfrac{\alpha \ell_{\alpha}+\beta 
\ell_{\beta}}{\ell_\alpha+\ell_\beta}
$
if and only if
\begin{itemize}
\item[(ia)] $n_0=1$, and either $\gamma_\varepsilon(p)=\beta$, 
$\gamma_\varepsilon(q)=\alpha$, or with an endpoint on 
$\mathcal{I}_{\alpha, 
\beta}$  ;
\item[(ib)] $n_0=-1$, and either $\gamma_\varepsilon(p)=\alpha$, 
$\gamma_\varepsilon(q)=\beta$, or with an endpoint on 
$\mathcal{I}_{\beta,\alpha}$ .
\end{itemize}
\item[(ii)] If $\ell\geq  \varepsilon$, $L$ is calibrable 
with velocity 
$
v_L=\dfrac{\alpha+\beta}{2}
$
if and only if
\begin{itemize}
\item[(iia)] $n_0=1$, and  $p,q \in \mathcal{I}_{\alpha, \beta}$;
\item[(iib)] $n_0=-1$, and $p,q \in \mathcal{I}_{\beta,\alpha}$.
\end{itemize}
\end{itemize}

\end{prop}

\begin{proof}
If $n_0=1$ the Cahn--Hoffman vector field $n$ needs to be neither increasing 
near $p$ nor decreasing near $q$. This occurs only if $L$ is the union of three 
consecutive segments
$L= L_\beta \cup L_c \cup  L_\alpha$, with endpoints in $p$, $p+\ell_\beta \in 
\mathcal{I}_{\beta,\alpha}$, $q-\ell_\alpha \in \mathcal{I}_{ 
\beta,\alpha}$, and $q$. If $L_c=\emptyset$, then
the candidate field \eqref{f:CHespl} always 
satisfied the constraint in Definition\ref{d:cali}(ii), proving (i). 

If $L_c\neq\emptyset$, from \eqref{f:CHespl} we get
\[
n(p+\varepsilon)-n(p)=\frac{\varepsilon}{2\ell}(\beta-\alpha)
(\ell_\beta-\ell_\alpha)=n(q)-n(q-\varepsilon),
\]
and hence,
since $n(p)=n(q)=1$, the constraint $|n|\leq 1$ is not satisfied if
$\ell_\alpha \neq \ell_\beta$. Finally,
if  $\ell_\alpha=\ell_\beta$, then, by \eqref{f:conden3}
\[
  n'(x)= 
  \begin{cases}
  \dfrac{\beta-\alpha}{2}
  &  \text{if}\ 
  \gamma_\varepsilon(x)=\alpha , \\[10pt]
  \dfrac{\alpha-\beta}{2}
   & \text{if}\ 
  \gamma_\varepsilon(x)=\beta,
  \end{cases}
\]
and a Canh--Hoffmann vector field with this derivative exists
only if  $\ell_\alpha=\ell_\beta=\varepsilon/2$,
otherwhise
$
 n(p+\ell_\beta+\varepsilon/2)>1.
$
In conclusion, $L$ is calibrable with velocity $v_L=(\alpha+\beta)/2$
if and only if $p,q\in \mathcal{I}_{\alpha, \beta}$.

The case $n_0=-1$ follows 
from similar arguments.
\end{proof}

\begin{prop}\label{p:fracon} 
Let $L$ be an edge with positive $\varphi$--curvature. If either 
\begin{equation}\label{f:fracon}
\ell+\ell_\alpha-\ell_\beta \leq 
\frac{4}{\beta-\alpha}
\end{equation}
or $p\in 
\mathcal{I}_{\beta,\alpha}$, 
$q \in \mathcal{I}_{\alpha,\beta}$, then $L$
%
is calibrable with velocity $v_L$ given by \eqref{f:velocita}.
\end{prop}

\begin{proof}
Under the assumption \eqref{f:fracon}, the candidate Cahn--Hoffmann vector 
field \eqref{f:CHespl} is an increasing function in $[p,q]$. Hence the 
constraint in Definition \ref{d:cali}(ii) is fulfilled, and $L$ is calibrable.

Assume now that $p\in 
\mathcal{I}_{\beta,\alpha}$, and 
$q \in \mathcal{I}_{\alpha,\beta}$. Then
the edge has $n(p)=-1$, $n(q)=1$,
$\ell_\alpha=\varepsilon/2$, and $\ell_\beta=0$. Then, by \eqref{f:conden3}, 
the candidate Cahn-Hoffmann field \eqref{f:CHespl} is increasing in 
$[p,p+\varepsilon/2]$, and, by Remark \ref{f:asseps}, 
\[
n(p+\varepsilon)-n(p)=\frac{\varepsilon}{4\ell} 
\left(8-(\beta-\alpha)\varepsilon\right) >0.
\]
Similarly, we obtain that $n$ satisfies the constraint in 
$[q-\varepsilon,q]$, and hence
$|n|\leq 1$ on $L$, so that $L$ is calibrable 
with velocity $v_L$.
\end{proof}

\begin{rk}\label{velfrac}
Notice that, if $L$ satisfies the condition \eqref{f:fracon}, then
\[
v_L\geq \frac{2}{\ell}+\frac{\alpha+\beta}{2}+ 
+\frac{\beta-\alpha}{2\ell}
\left(\ell-\frac{4}{\beta-\alpha}\right)=\beta>0
\]
\end{rk}



\begin{prop}\label{p:gencalibr}
Let $L$ be an edge with positive 
$\varphi$--curvature, and such that $\ell+\ell_\alpha-\ell_\beta > 
4/(\beta-\alpha)$. Then the following hold.
\begin{itemize}
\item[(i)] If either $\gamma_\varepsilon (p)=\beta$, or $\gamma_\varepsilon 
(q)=\beta$, or $p\in \mathcal{I}_{\alpha,\beta}$, or
$q\in \mathcal{I}_{\beta,\alpha}$, 
then $L$ is not calibrable.
\item[(ii)] If $\gamma_\varepsilon (p)=\gamma_\varepsilon (q) =\alpha$, let 
$\sigma_1$, $\sigma_2\in (0, \varepsilon/2)$ be such that 
$p+\varepsilon/2+\sigma_1 \in \mathcal{I}_{\beta,\alpha}$ and 
$q-\varepsilon/2-\sigma_2 \in  \mathcal{I}_{\alpha,\beta}$, and 
let $\tilde{\ell}$ be the length of the interval $[p+\varepsilon/2+\sigma_1,
q-\varepsilon/2-\sigma_2]$.
Setting 
\[
m= \varepsilon 
\frac{\beta-\alpha}{(\beta-\alpha)(\tilde{\ell}+\varepsilon/2)+4},
\qquad 
h=\frac{\varepsilon}{2}\frac{(\beta-\alpha)(\tilde{\ell}+\varepsilon/2)-4}
{(\beta-\alpha)(\tilde{\ell}+\varepsilon/2)+4},
\] and  
\[
\Sigma=\left\{m\sigma_2 +h \leq  
\sigma_1 \leq 
\frac{1}{m} \sigma_2- \frac{h}{m}\right\},
\]
we have $m\in (0,1)$, $\Sigma \cap [0,\varepsilon/2]^2 \neq \emptyset$, and 
$L$ is calibrable with velocity
\[
v_L=\frac{2}{\ell}+\frac{\alpha+\beta}{2}+\frac{\beta-\alpha}{2\ell}
\left(\frac{\varepsilon}{2}-\sigma_1-\sigma_2\right)
\]
if and only if $(\sigma_1,\sigma_2)\in \Sigma$.
\item[(iii)] if $\gamma_\varepsilon(p)=\alpha$, and $q\in  
\mathcal{I}_{\alpha,\beta}$ (resp. $p\in \mathcal{I}_{\beta, 
\alpha}$, and $\gamma_\varepsilon(q)=\alpha$), let $\sigma\in (0,\varepsilon/2)$
be such that $p+\sigma+\varepsilon/2\in \mathcal{I}_{\beta,\alpha}$
(resp. $q-\sigma-\varepsilon/2\in \mathcal{I}_{\alpha,\beta}$),
let $\ell^*$ be the length of the interval $[p+\varepsilon/2+\sigma, q]$ (
resp. of $[p, q-\varepsilon/2-\sigma]$), and let
\[
\sigma^*= \frac{\varepsilon}{2}\frac{(\beta-\alpha)(\ell^*+\varepsilon/2)-4}
{(\beta-\alpha)(\ell^*-\varepsilon/2)+4}.
\]
Then $L$ is calibrable if and only if $\sigma \geq \sigma^*$.
\end{itemize}
\end{prop}

\begin{proof}
If $\ell+\ell_\alpha-\ell_\beta > 4/(\beta-\alpha)$, by \eqref{f:conden3}
the candidate Cahn--Hoffmann field $n$ is strictly decreasing in the $\beta$ 
phase. Hence, under the assumptions in (i), 
$n$ does not satisfy the constraint $|n|\leq 1$ at least near an 
endpoint, and $L$ is not 
calibrable.

If both the endpoints belong to the $\alpha$ phase, then the requirement
$(\sigma_1,\sigma_2)\in\Sigma$ is equivalent to the conditions
\[
\begin{cases}
n(p+\varepsilon/2+\sigma_1)-n(p)\geq 0,\\
n(q-\varepsilon/2-\sigma_2)-n(q)\geq 0
\end{cases}
\]
that guarantee the calibrability of the edge.

Setting
\begin{equation}\label{f:tildesig}
\tilde{\sigma}:=
\frac{\varepsilon}{2}\frac{(\beta-\alpha)(\tilde{\ell}+\varepsilon/2)-4}
{(\beta-\alpha)(\tilde{\ell}-\varepsilon/2)+4},
\end{equation}
we have that $(\tilde{\sigma},\tilde{\sigma})\in \Sigma$,
and $\tilde{\sigma}\in (0,\varepsilon/2)$ under the assumption 
$\ell+\ell_\alpha-\ell_\beta > 4/(\beta-\alpha)$.

The proof of (iii) follows the same arguments. 
\end{proof}

\begin{rk}[Calibrability threshold]\label{r:symm}
In the special case when $\sigma_1=\sigma_2=\sigma>0$, the calibrability 
condition stated in Proposition 
\ref{p:gencalibr}(ii) reduces to the unilateral constraint
$\sigma\geq \tilde{\sigma}$, where $\tilde{\sigma}$ is the value defined in 
\eqref{f:tildesig}.
Hence $L$ is calibrable if and only if $\sigma\geq \tilde{\sigma}$.
Moreover, if $\sigma=\tilde{\sigma}$, the edge $L$ is calibrated by a 
Cahn--Hoffmann
vector field $n$ such that $n(p)=n(p+\varepsilon/2+\tilde{\sigma})$ and 
$n(q)=n(q-\varepsilon/2-\tilde{\sigma})$. As a consequence,
the same field calibrates both the edges 
$[p,p+\varepsilon/2+\tilde{\sigma}]\times 
\{\overline{y}\}$
and $[q-\varepsilon/2-\tilde{\sigma},q]\times \{\overline{y}\}$
(as edges with zero $\varphi$--curvature, see Proposition \ref{r:hcurvz}),
and the edge 
$[p+\varepsilon/2+\tilde{\sigma},q-\varepsilon/2-\tilde{\sigma}]\times 
\{\overline{y}\}$ (as edges with positive $\varphi$--curvature, see Proposition 
\ref{p:fracon}) with the same velocity.

Similarly, in the case (iii) Proposition \ref{p:gencalibr}, when 
$\sigma=\sigma^*$, the edge $L$ is calibrated by a Cahn--Hoffmann
vector field $n$ such that $n(p)=n(p+\varepsilon/2+\sigma^*)$, and 
the same field calibrates both the edge 
$[p,p+\varepsilon/2+\sigma^*]\times 
\{\overline{y}\}$
(as edge with zero $\varphi$--curvature),
and the edge 
$[p+\varepsilon/2+\tilde{\sigma},q]\times 
\{\overline{y}\}$ (as edges with positive $\varphi$--curvature) with the same 
velocity.
\end{rk}

\section{Forced crystalline flows and their effective motion}\label{s:motion}

The results of Section \ref{s:calibr} prescribe a velocity to every 
calibrable edge
not lying on a discontinuity line of $g_\eps$, and suggest that the forced 
crystalline  
curvature flow starting from a coordinate polyrectangle (that is a set 
whose boundary is a closed polygonal curve with edges parallel to the 
coordinate axes) remains a coordinate polyrectangle, whose structure changes
when either existing edges disappear by the growth of their 
neighbors, 
or new edges are generated by the splitting of no longer calibrable edges.

In every time interval between these events, the motion is determined 
by a system of ODEs, and hence the behavior of the evolution 
on the discontinuities can be described using  the general theory of 
differential equations with discontinuous 
right--hand side \cite{Fi}.

Concerning the changes of geometry, it is clear what is meant by
``disappearing edges'', that is edges whose length becomes zero in finite 
time, but the notion of ``appearing edges'', that is how a no longer 
calibrable edge breaks, has to be specified. 

\smallskip

We focus our attention to 
coordinate polyrectangles whose edges have non--negative $\varphi$--curvature.
In the sequel we will use the abuse of notation $L=[p,q]$ when
the edge $L$ is of the form $L=[p,q]\times \{y\}$ or $L=\{y\} \times [p,q]$,
and $n(p)$, $n(q)$ will denote the prescribed values of the Cahn--Hoffman 
vector field at the endpoints of $L$. 

\begin{Definition}[Cracking multiplicity and set--up]\label{d:crsu}

If $L$ is an edge not lying on a discontinuity line of $g_\eps$,
let us define
\[
L^c:=\sup \{\tilde{L}\subseteq L\colon \ \tilde{L}=[s_1,s_2], \ n(s_1)=n(p), 
n(s_2)=n(q),\ \tilde{L}
\ \text{calibrable}\} =[p_b,q_b],
\]
and let us denote by $L^-=[p,p_b]$, $L^+=[q_b,q]$, with the convention
$L^-=\emptyset$ (resp. $L^+=\emptyset$) if $p=p_b$ (resp. $q=q_b$).  
The {\em craking multiplicity}
$M(L)$ assigned to $L$ is given by 
\[
M(L):= 
\begin{cases}
1, & \text{if $L=L^c$}, \\
3, & \text{if $L\neq L^c$, and either $L^-= \emptyset$, or $L^+=\emptyset$, } \\
5,  & \text{if  $L^-\neq  \emptyset$, and $L^+\neq \emptyset$. }
\end{cases}
\]
The points $p_b$, $q_b$, if different from the endpoints of $L$, are said
{\em breaking points} of $L$, and $\mathcal{C}(L):=\{p,p_b,q_b,q\}$
is the {\em craking set--up} of $L$.

\smallskip

For every edge $L$ lying on a discontinuity line of $g_\eps$, with (inner) 
normal $\nu(L)$, consider the values 
\begin{equation}\label{f:inoutv}
v^{in}_L=\chi_{L}\frac{2}{\ell}+ 
\frac{1}{\ell} 
\int_{L+\frac{\eps}{4}\nu(L)}
g_\eps, \qquad
v^{out}_L=\chi_{L}\frac{2}{\ell}+ 
\frac{1}{\ell} 
\int_{L-\frac{\eps}{4}\nu(L)}
g_\eps.
\end{equation} 
If  $v_L^{in}> 0$ and $v_L^{out}<0 $, then we set 
$M(L)\in \{1, M(L+\frac{\eps}{4}\nu(L)),M(L-\frac{\eps}{4}\nu(L))$. Otherwise,
the craking multiplicity $M(L)$ assigned to $L$ is given by
\[
M(L)= 
\begin{cases}
M(L+\frac{\eps}{4}\nu(L)) & \text{if $v_L^{in}> 0$ and $v_L^{out}\geq   0$} \\
M(L-\frac{\eps}{4}\nu(L))  & \text{if $v_L^{in}\leq 0$ and $v_L^{out}<   0$} \\
1  & \text{if $v_L^{in}\leq 0$ and $v_L^{out}\geq   0$} 
\end{cases}
\]
When $M(L)=M(L\pm\frac{\eps}{4}\nu(L))$, the cracking set--up of $L$
is set as $\mathcal{C}(L):= \mathcal{C}(L\pm\frac{\eps}{4}\nu(L))$.
\end{Definition}

\begin{prop}\label{r:breakablepos}
Let $L=[p,q]$ be an edge not lying on a discontinuity line of $g_\eps$,
and let $\mathcal{C}(L):=\{p,p_b,q_b,q\}$ the cracking set--up of $L$. Then
$L^-$ and $L^+$ are either empty, or calibrable as edges with zero 
$\varphi$--curvature and prescribed Cahn--Hoffman vector field $n(p)=n(p_b)$, 
and $n(q_b)=n(q)$, respectively.
Moreover, denoting by $v^\pm, v^c$ the velocities of $L^\pm,L^c$ 
respectively, then 
\begin{itemize}
\item[(i)] $p_b\in \mathcal{I}_{\beta,\alpha}$ , 
$q_b \in \mathcal{I}_{\alpha,\beta,}$, and $v^\pm>v^c$, if $L$ has positive 
$\varphi$--curvature;
\item[(ii)] $p_b,q_b\in \mathcal{I}_{\beta,\alpha}$, and $v^->v^c>v^+$, if $L$ 
has zero $\varphi$--curvature, 
and $n(p)=-1$;
\item[(iii)] $p_b,q_b\in \mathcal{I}_{\alpha,\beta}$, and $v^-<v^c<v^+$, if $L$ 
has 
zero $\varphi$--curvature, 
and $n(p)=1$;
\end{itemize}
\end{prop}

\begin{proof}
Let $L=[p,q]$ be an edge with positive 
$\varphi$--curvature. By Propositions \ref{p:fracon} and \ref{p:gencalibr},
if $L^c \neq L$ then $\ell+\ell_\alpha-\ell_\beta > 
4/(\beta-\alpha)$, and 
\[
p_b=\min\{s\in [p,q]\cap \mathcal{I}_{\beta,\alpha}\}, \qquad
q_b=\max\{s\in [p,q]\cap \mathcal{I}_{\alpha, \beta}\},
\]
Moreover, by Proposition \ref{r:hcurvz}, $L^-= [p,p_b]\subseteq L$
is either empty or calibrable
as edge with zero $\varphi$--curvature and prescribed Cahn--Hoffmann 
$n_0=-1$ at the endpoints. Similarly
$L^+= [q_b,q]\subseteq L$ is either empty  or calibrable
as edge with zero $\varphi$--curvature and prescribed Cahn--Hoffmann 
$n_0=-1$ at the endpoints.

Concerning the velocities, assume that $p\neq p_b$. If either
$\gamma_\eps(p)=\beta$ or $p\in \mathcal{I}_{\alpha,\beta}$, we have 
$v^-=\beta > v^c$, by Remark \ref{velfrac}. If $\gamma_\eps(p)=\alpha$,
then, by Proposition \ref{r:hcurvz}(i), 
\[
v^-=v(\sigma)=\frac{ \sigma \alpha +\frac{\eps}{2} \beta}{ \sigma 
+\frac{\eps}{2}} 
\]
with $0<\sigma<\sigma_0$, where $\sigma_0$ is such that
$v(\sigma_0)=v^c$ (see also Remark \ref{r:symm}). The inequality
$v^->v^c$ then follows from the fact that the function
$v(\sigma)$ is strictly monotone decreasing. The same arguments
can be used for $v^+$ in the case $q_b\neq q$. 

\smallskip

Recalling Proposition \ref{r:hcurvz}, we can perform a similar
splitting for edges with zero $\varphi$--curvature. If $q-p\geq \eps$, then, 
if $n(p)=n(q)=1$, we have
\[
p_b=\min\{s\in [p,q]\cap \mathcal{I}_{\alpha,\beta}\}, \qquad
q_b=\max\{s\in [p,q]\cap \mathcal{I}_{\alpha, \beta}\}, 
\]
while, if $n(p)=n(q)=-1$, we have
\[
p_b=\min\{s\in [p,q]\cap \mathcal{I}_{\beta, \alpha}\}, \qquad
q_b=\max\{s\in [p,q]\cap \mathcal{I}_{\beta, \alpha}\}.
\]
In both cases the remaining parts $L^\pm$ are either empty or calibrable 
with velocities
\[
v^\pm= \frac{\sigma_\alpha^\pm \alpha+\sigma_\beta^\pm \beta}{\sigma_\alpha^\pm 
+\sigma_\beta^\pm }, 
\]
for suitable $\sigma_\alpha^\pm,\sigma_\beta^\pm \in [0, \eps/2]$. 
Since $v^c=\frac{\alpha+\beta}{2}$, the strict inequalities in 
(ii) and (iii) hold true.

The case of $L$ with zero $\varphi$--curvature and $q-p< \eps$ is similar and 
is left to the reader. We would like to stress that the edges of this type
appearing as $L^\pm$, due to the cracking set--up, are always calibrable.
\end{proof}
   

\begin{Definition}[Breaking configuration]
Let $E$ be a coordinate polyrectangle whose edges
$\tilde{L}_1,\ldots,\tilde{L}_n$ have non--negative $\varphi$--curvature.  
For every $i=1,\ldots,n$, let
$\mathcal{C}(L_i)=\{p_i,p_{i,b},q_{i,b},q_{i}\}$ be a cracking set--up of the 
edge $\tilde{L}_i$. 
The {\em breaking configuration} of $\partial E$ associated to 
$\{\mathcal{C}(\tilde{L}_i)\}_{i=1}^n$ is given by $L_1,\ldots,L_m$, $m= 
\sum_{i=1}^n M(\tilde{L}_i)$, where $L_j$ is either a part of and edge
$\tilde{L}_i$ obtained  by the splitting procedure in 
Definition \ref{d:crsu}, or a degenerate segment with legth zero at a
cracking point $p_{i,b}$ or $q_{i,b}$. 
\end{Definition}

\begin{rk}\label{r:bcuniq}
By Proposition \ref{r:breakablepos}, a breaking configuration of the 
boundary of a coordinate polyrectagle $E$, whose edges have non--negative 
$\varphi$--curvature always exists.
Moreover,  
since the cracking multiplicity and set--up of an edge $L$ is unambiguous 
except in 
the case 
$v_L^{in}> 0$ and 
$v_L^{out}<0 $, the breaking configuration  
is unique, provided that no edge $L\subseteq \partial E$ has $v_L^{in}> 0$ and 
$v_L^{out}<0 $.
\end{rk}

The following result shows that, in our setting, the evolution is well posed.

\begin{prop}\label{p:pinning}
Let $E$ be a coordinate polyrectangle whose edges have non--negative 
$\varphi$--curvature, and let $L_1,\ldots,L_m$ a breaking 
configuration of $\partial E$.
Then there exists $T>0$ and a  family $E(t)$, $t\in [0,T]$ of
coordinate polyrectangles with edges
$L_1(t),\ldots,L_{m}(t)$ which is a forced crystalline flow 
starting from $E$. If, in addition, every $L_i$ with positive length
and
lying on a 
discontinuity line of $g_\eps$ satisfies
one of the following properties
\begin{itemize}
\item[(1)] $v^{in}_{L_i}< 0$, and $v^{out}_{L_i}>0$,
\item[(2)] $v^{in}_{L_i}> 0$ and  $v^{out}_{L_i} > 0$,
\item[(3)] $v^{in}_{L_i}< 0$ and  $v^{out}_{L_i} < 0$,
\end{itemize}
then the 
evolution is unique.
Moreover if $L_i$ satisfies condition (1), 
then $L_i(t)$ is pinned until $v^{in}_{L_i(t)}\leq  0$ and 
$v^{out}_{L_i(t)}\geq 0$.
\end{prop}

\begin{proof}
A given coordinate polyrectangle $E$, with edges
$\tilde{L}_1,\ldots,\tilde{L}_n$, is completely
detemined by the strings $(\tilde{\nu}_1,\ldots,\tilde{\nu}_n)$ and 
$(\tilde{s}_1,\ldots,\tilde{s}_n)$,
where $\tilde{\nu}_i$ and $\tilde{s}_i$ are the inner normal vector and the 
distance from the 
origin of the edge $L_i$, $i=1,\ldots,n$, respectively. 

Let us denote $e_1=e_5=(1,0)$, $e_2=(0,1)$, $e_3=(-1,0)$, and $e_4=e_0=(0,-1)$,
so that if  $\tilde{\nu}_i = e_j$, then $\tilde{\nu}_{i+1} \in 
\{e_{j-1},e_{j+1}\}$.

We associate to $E$ a breaking 
configuration $(\nu_1,\ldots,\nu_m)$, $(s_1,\ldots,s_m)$  based on the cracking 
set--up of its edges in the following way. 

If $M(L_1)=1$, then  $\nu_1=\tilde{\nu}_1$ and $s_1=\tilde{s}_1$. If 
$M(L_1)=3$, $q_{1,b}=q_1$, and $\nu_1=e_j$  then 
$(\nu_1,\nu_2,\nu_3)= (e_j,e_{j+1},e_j)$ and 
$(s_1,s_2,s_3)=(\tilde{s}_1, p_{1,b}, 
\tilde{s}_1)$. Similarly, if $M(L_1)=3$, $p_{1,b}=p_1$, and $\nu_1=e_j$  then 
$(\nu_1,\nu_2,\nu_3)= (e_j,e_{j-1},e_j)$ and 
$(s_1,s_2,s_3)=(\tilde{s}_1, q_{b,1}, 
\tilde{s}_1)$. Finally, if $M(L_1)=5$, then 
$(\nu_1,\nu_2,\nu_3,\nu_4,\nu_5)= (e_j,e_{j+1},e_j,e_{j-1},e_j)$ and 
$(s_1,s_2,s_3,s_4,s_5)=(\tilde{s}_1, p_{1,b}, \tilde{s}_1, q_{1,p}, 
\tilde{s}_1)$. The subsequent elements of the strings are obtained applying the 
same procedure to $L_2$, and so forth. 

By Definition \ref{d:flow}, a forced crystalline flow 
$E(t)$, $t\in [0,T]$ given by calibrable coordinate polyrectangles with 
edges $L_1(t),\ldots,L_m(t)$ and
with normal direction $(\nu_1,\ldots,\nu_m)$, is then identified with a 
solution 
$s(t)=(s_1(t),\ldots,s_m(t))$  to the system of ODEs
\begin{equation}\label{f:dissyst}
s'=V(s) \qquad \text{in}\ [0,T]
\end{equation}
where $V \colon \R^m \to \R^m$ is the field, discontinuous on 
\[
\Sigma=\left\{s\in \R^m \colon \ \exists s_i\in 
\frac{\eps}{2}\Z\right\},
\]
and defined outside $\Sigma$ by
\[
V=(V_1,\ldots,V_m),\qquad  V_i(s)=-\left(\chi_{L_i}\frac{2}{\ell_i(s)}+ 
\frac{1}{\ell_i(s)} 
\int_{L_i(s)} 
g_\eps\right),\quad i=1,\ldots, m, \ s\not\in \Sigma.
\]

Notice that the fictitious edges with zero length, possibly added in the 
breaking
configuration of $E$, are contained on discontinuity lines of $g_\eps$. 
Then either $E$ is calibrable, or it corresponds to a string $s\in\Sigma$.

System \eqref{f:dissyst} fits therefore into Filippov's theory of
discontinuous dinamical systems (see \cite{Fi}, \cite{Co}): the field $V$
is extended on $\Sigma$ by the multifunction
\begin{equation}\label{f:mult}
F(V)(s)= \text{co}\left\{\lim_{k\to \infty}V(s^k), s^k \to s, 
\ s_k \not\in \Sigma \right\}, \qquad s \in\Sigma,
\end{equation}
(where we denote by $\text{co}(A)$ the convex envelope of a set $A$) 
and a solution of \eqref{f:dissyst} is, by definition, a solution of 
the differential inclusion $s'\in F(V)(s)$. Since $V\colon \R^m \to \R^m$
is measurable and essentially bounded, then there exists at least a solution
of such a differential inclusion, starting from any initial datum $s$.

In order to deal with the uniqueness of solutions, we need an explicit 
computation of the multifunction $F(V)$ on $\Sigma$. 

For every $s \in \Sigma$, and for every component $s_i$ of $s$ such that
$s_i\in \frac{\eps}{2}\Z$, and $s_{i-1}\neq s_{i+1}$, so that
$\ell_i(s)>0$, let $V_i^+(s)$ and 
$V_i^+(s)$ be the values
\[
\begin{split}
V_i^+(s) & =
-\left(\chi_{L_i}\frac{2}{\ell_i(s)}+ 
\frac{1}{\ell_i(s)} 
\int_{L_i(s)+\frac{\eps}{4}\nu_i} 
g_\eps\right), \\
V_i^-(s) & =-\left(\chi_{L_i}\frac{2}{\ell_i(s)}+ 
\frac{1}{\ell_i(s)} 
\int_{L_i(s)-\frac{\eps}{4}\nu_i} 
g_\eps\right),
\end{split}
\]
and let $I(V_i^-(s),V_i^+(s))$ be the interval with endpoints 
$V_i^-(s)$ and $V_i^+(s)$.  

For every $i=1,\ldots,m$, 
$V_i(s)$ depends only on $s_{i-1}$, $s_{i}$, and $s_{i+1}$, and it is 
discontinuous only in the $s_{i}$ variable, 
then $F(V)(s)$ in \eqref{f:mult} is 
the 
convex set
\[
F(V)(s)= I_1(s) \times \cdots \times I_m(s)
\]
where
\[
I_i(s)=
\begin{cases}
\{V_i(s)\}, & \text{if}\ s_i\not\in \frac{\eps}{2}\Z, \\
I(V_i^-(s),V_i^+(s)) & \text{if}\ s_i\in \frac{\eps}{2}\Z, \text{and}\ 
\ell_i>0\\
[\alpha,\beta] & \text{if}\ s_i\in \frac{\eps}{2}\Z, \text{and}\ \ell_i=0,
\end{cases}
\qquad i=1,\ldots,m.
\]
Assume now that every element of the breaking configuration of $\partial E$
with positive length satisfies one of the conditions
(1), (2), (3). 

Concerning the edges with zero length, notice that 
if $L_i(t)$ is an edge starting from $L_i$ with $\ell_i=0$, then
$\ell_i(t)>0$  and  $v^{in}_{L_i(t)}=\alpha$, $v^{out}_{L_i(t)}=\beta$ for $t>0$
small enough, so that the edge 
$L_i(t)$ fulfills (1). 
Hence we can split the indices $\{1,\ldots,m\}=
\mathcal{N}_1\cup\mathcal{N}_2\cup\mathcal{N}_3$ in such a way
$L_i(t)$ satisfies (j) for every $i\in \mathcal{N}_j$, locally near $t=0$.

Let $s^0$ be the string corresponding to the breaking configuration of
$\partial E$, let $E(t)$ be any
evolution obtained by solving the differential inclusion \eqref{f:dissyst}
with initial datum $s(0)=s^0$, and let $\overline{t}\in (0,T]$ be such that
$L_i(t)$ satisfies (j) in $(0,\overline{t})$ for every $i\in \mathcal{N}_j$, 
$j=1,2,3$. 
The family $E(t)$ is then identified with $s(t)=(s_1(t), \ldots,s_m(t))$
solving $s_i'(t)\in I_i(s(t))$ a.e. in $(0,\overline{t})$, $i=1,\dots,m$. 

Then, for every $i\in \mathcal{N}_1$, we have that 
$s_i'(s_i-s_i(0)) \leq 0$ for every choice of 
$s_i'(t)\in 
I_i(s(t))=I(V_i^-(s(t)),V_i^+(s(t)))$, and hence $s_i'=0$ in 
$(0,\overline{t})$.  
 (see, e.g., 
\cite{Fi} Corollary 2.10.2).

On the other hand, for every $i\in \mathcal{N}_2\cup  \mathcal{N}_3$, then 
either $s_i'>0$ or $s_i'< 0$  in $(0,\overline{t})$, and hence 
$s_i(t)=\{V_i(s(t))\}$ a.e. in $(0,\overline{t})$.

In conclusion, since the function $V$ is Lipschitz continuous outside 
$\Sigma$,  the solution of the differential inclusion is unique 
in $(0,\overline{t})$, and it is fully determined by the law
\[
s'_i =
\begin{cases}
0 & i\in \mathcal{N}_1 \\
V_i(s) & i\in \mathcal{N}_2\cup  \mathcal{N}_3.
\end{cases} 
\qquad \text{a.e. in} \ (0,\overline{t}).   
\] 
In terms of the breaking configuration of $\partial{E}$, we can conclude 
that every edge $L_i$, with $i\in \mathcal{N}_2\cup  \mathcal{N}_3$ crosses 
immediately the discontinuity line, moving
inward (respectively outward) if $i\in \mathcal{N}_2$ (respectively 
$i\in \mathcal{N}_3$),
while every
edge $L_i$ with $i\in \mathcal{N}_1$ is pinned on the discontinuity 
line.
\end{proof}

\begin{rk}
\label{r:equilibria}
As a consequence of Proposition \ref{p:pinning}, if $\alpha+\beta<0$, for every
$\eps>0$ there are 
nontrivial equilibria of the forced crystalline curvature flow. For example, 
a calibrable coordinate
polyrectangle $E$ such that
\begin{itemize}
\item[(a)] every vertex of $E$ 
is also a
vertex of of a square $Q \in \acube$, $Q\subseteq E$,
\item[(b)] every edge of $\partial E$ with zero $\varphi$--curvature has 
length $\ell=\eps/2$,
\item[(c)] every edge of $\partial E$ with positive $\varphi$--curvature has 
length $\ell$ very closed to $-4/(\alpha+\beta)$,
\end{itemize}
is pinned. Namely, requirement (a) implies that every edge of $\partial E$ 
lies on a 
discontinuity line, (b) guarantees that $v^{in}_L=\alpha$ and 
$v^{out}_L=\beta$ for 
every 
edge $L$ with zero $\varphi$--curvature, while (c) guarantees that 
\[
v^{in}_L= \frac{2}{\ell} +\frac{\alpha+\beta}{2}- 
\frac{(\beta-\alpha)\varepsilon}{4 \ell}<0, \quad 
v^{out}_L= \frac{2}{\ell} +\frac{\alpha+\beta}{2}+ 
\frac{(\beta-\alpha)\varepsilon}{4 \ell}>0.
\]
for every edge $L$ with positive $\varphi$--curvature.
   
In particular, the symmetric equilibria $O_\eps$ (see Figure \ref{fig:rhombus})
converge, as $\eps \to 0$ to an octagon $O$ having horizontal and vertical 
edges with length $\ell=-4/(\alpha+\beta)$, connected by diagonal edges.
\end{rk}

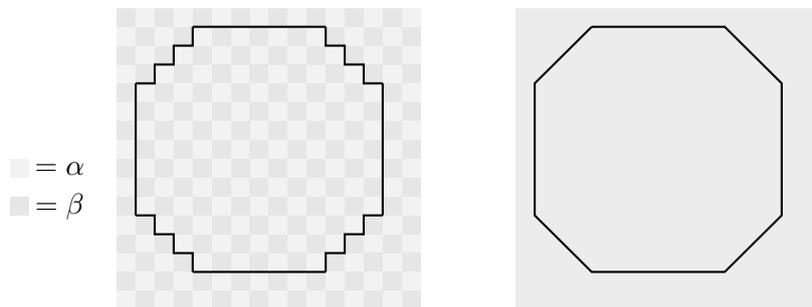
\begin{figure}[h!]
\centering
\begin{tikzpicture}[scale=0.5]
\foreach \x in {1,2,...,8}
\foreach \y in {1,2,...,8}
{
\filldraw[fill=gray!10!white,draw=none] (\x,\y) rectangle +(.5,.5);
\filldraw[fill=gray!10!white,draw=none] (\x,\y) rectangle +(-.5,-.5);
\filldraw[fill=gray!20!white,draw=none] (\x,\y) rectangle +(.5,-.5);
\filldraw[fill=gray!20!white,draw=none] (\x,\y) rectangle +(-.5,.5);
}

\filldraw[fill=gray!10!white,draw=none] (-2.3,4) rectangle +(.5,.5);
\node at (-1,4.25) {$=\alpha$};
\filldraw[fill=gray!20!white,draw=none] (-2.3,3) rectangle +(.5,.5);
\node at (-1,3.25) {$=\beta$};

\draw[thick] (2.5,8) -- (6,8);
\draw[thick] (1,3) -- (1,6.5);
\draw[thick] (7.5,3) -- (7.5,6.5);
\draw[thick] (2.5,1.5) -- (6,1.5);
\draw[thick] (1,3) -- (1.5,3) -- (1.5,2.5) -- (2,2.5) -- (2,2) -- (2.5,2) -- 
(2.5,1.5);
\draw[thick] (1,6.5) -- (1.5,6.5) -- (1.5,7) -- (2,7) -- (2,7.5) -- (2.5,7.5) 
-- (2.5,8);
\draw[thick] (6,1.5) --(6,2) -- (6.5,2) -- (6.5,2.5) -- (7,2.5) -- (7,3) -- 
(7.5,3);
\draw[thick] (7.5, 6.5) -- (7,6.5) -- (7,7) -- (6.5,7) -- (6.5,7.5) -- (6,7.5) 
-- (6,8);

\filldraw[fill=gray!15!white,draw=none] (11,0.5) rectangle +(8,8);
\draw[thick] (13,8) -- (16.5,8) -- (18,6.5) -- (18,3) -- (16.5,1.5) --
(13,1.5) -- (11.5,3) -- (11.5,6.5) -- (13,8)  ;
\end{tikzpicture}
\caption{Microscopic and macroscopic nontrivial equilibrium ($\alpha+\beta<0$).}
\label{fig:rhombus}
\end{figure}

In conclusion, the forced crystalline evolutions defined in Definition 
\ref{d:flow} and starting from a polyrectangle are obtained by the following 
procedure: we set--up the initial datum and we obtain the evolution $E(t)$ 
by solving the system of ODEs \eqref{f:dissyst}, for $t\in [0,T]$, where $T>0$ 
is the first time when an edge either disappear or is no more 
calibrable. The subsequent evolution is obtained by inizializing $E(T)$ 
according to Definition \ref{d:crsu} as an initial datum for the new system of 
ODEs of the form \eqref{f:dissyst}. 

\smallskip

We are interested in stressing the macroscopic effect of the underling periodic 
structure on the geometric evolutions, depicting clearly the forced crystalline 
flows and passing to the limit as $\eps \to 0$,  and the most of the features 
are revealed by the evolution starting from the simplest crystals: the 
coordinate squares.

In what follows $S(\ell)$ will denote a coordinate square with side length 
$\ell>0$.

\begin{teor}[Effective motion of coordinate squares]\label{d:square}
Let $S(\ell_0)$ be a given coordinate square. For every $\eps>0$, let 
$S(\ell_0^\eps)$ be a
coordinate square  such that
$
d_H(S(\ell_0),S(\ell_0^\eps))<\varepsilon. 
$
Then there exists a forced crystalline curvature flow $E^\eps(t)$, 
$t\in [0,T)$,
starting from $S(\ell_0^\eps)$. Moreover, there exists a family of 
sets $E(t)$, $t\in [0,T)$, such that  $E^\eps(t)$ converges to $E(t)$ in the 
Hausdorff
topology and locally uniformly in time, as $\varepsilon \to 0$.
The limit evolution $E(t)$ is independent of
the choice of the approximating initial data  $S(\ell_0^\eps)$, but its 
geometry depends on $\ell_0$ in the following way.
\begin{itemize}
\item[(i)] If either $\alpha+\beta\geq 0$ and $\ell_0>0$ or $\alpha+\beta<0$ 
and $0<\ell_0\leq -4/(\alpha+\beta)$, then $E(t)$ is a family of coordinate 
squares 
$E(t)= 
S(\ell(t))$, with
$\ell(t)$ governed by the ODE
\begin{equation}\label{f:shrss}
\begin{cases}
\ell'=-\dfrac{4}{\ell}-(\alpha+\beta), \\
\ell(0)=\ell_0,
\end{cases}
\end{equation}
and then shrinking to a point in finite time.
\item[(ii)] If $\alpha+\beta<0$ 
and $\ell_0> -4/(\alpha+\beta)$, then 
$E(t)$ is a family of octagons $E(t)=S(\ell_0)^\circ \cap S(\tilde{\ell}(t))$, 
where $S(\ell_0)^\circ$ is the polar square of $S(\ell_0)$, and 
$\tilde{\ell}(t)$ is the
solution to \eqref{f:shrss}. In particular, the moving edges of $E(t)$ have 
length
governed by the ODE
\begin{equation}\label{f:octsq}
\begin{cases}
\ell'=\dfrac{4}{\ell}+(\alpha+\beta), \\
\ell(0)=\ell_0,
\end{cases}
\end{equation}
and then $E(t)$ is increasing in $[0,+\infty)$, and converging to a stationary 
octagon
as $t\to +\infty$.
\end{itemize}

\end{teor}

\begin{proof}

Given $S(\ell_0)$ and $\eps>0$, let $S(\ell_0^\eps)$ be  a square such that
$
d_H(S(\ell_0),S(\ell_0^\eps))<\varepsilon. 
$

\medskip

\noindent Case (i)a: $\ell_0 \leq 4/(\beta-\alpha)$ (self-similar 
shrinking). 

\smallskip

By Proposition \ref{p:fracon}, the breaking configuration of 
$\partial S(\ell_0^\eps)$ has no breaking points, and,
by Remark \ref{velfrac} either $v_L^{in},v_L^{out}> 0$, if 
$L\subseteq \partial S(\ell_0^\eps)$ is  on a 
discontinuity line of $g_\eps$, or $v_L> 0$, otherwise. Then, by Remark
\ref{r:bcuniq}, and Proposition \ref{p:pinning}, there exists a unique forced 
crystalline flow $E^\eps(t)$ starting from $S(\ell_0)$, and it is given by 
calibrable squares $S(\ell^\eps(t))$ with side 
length governed by the ODE
\begin{equation}\label{f:epsshrss}
{(\ell^\eps)}'=-\frac{4}{\ell^\eps}-(\alpha+\beta)-
\frac{\beta-\alpha}{\ell^\eps}(\ell_\beta^\eps-\ell_\alpha^\eps).
\end{equation}
Since $|\ell_\beta^\eps-\ell_\alpha^\eps|\leq \eps/2$, a passage to the 
limit in \eqref{f:epsshrss} as $\eps \to 0$ shows that 
$E^\eps(t)$ converges in the 
Hausdorff
topology and locally uniformly in time
to the family of squares $S(\ell(t))$ with side 
length governed by the ODE \eqref{f:shrss}.

\medskip

\noindent Case (i)b: 
either 
$\ell_0> 4/(\beta-\alpha)$ (if $\alpha+\beta\geq 0$), 
or 
$4/(\beta-\alpha) < \ell_0 
\leq -4/(\alpha+\beta)$ (if $\alpha+\beta< 0$) (shrinking with 
temporary breaking).

\smallskip

As a first step, we assume, in addition, that every vertex of 
$S(\ell_0^\eps)$ 
is also a
vertex of of a square $Q \in \acube$, $Q\subseteq S(\ell_0^\eps)$ 
(see Figure \ref{fig:brec}(I)), so that the edges lie on discontinuity
lines 
of $g_\eps$ and they have (the same) velocities  $v_{0,\eps}^{in}$, 
$v_{0,\eps}^{out}\geq 0$. Hence, by Definition \ref{d:crsu} and 
Proposition \ref{p:fracon}, we have 
\[
M(L^\eps)=M\left(L^\eps+ \frac{\eps}{4}\nu(L^\eps)\right)=1,
\qquad \forall L^\eps \subseteq \partial S(\ell_0^\eps),
\]
and, by Proposition \ref{p:pinning}, there exists a unique
forced crystalline flow starting from $S(\ell_0^\eps)$, given by squares 
$S(\ell^\eps(t))$ with side length governed by the ODE 
\eqref{f:epsshrss}, and defined in $(0,t_0)$ where
\[
t_0=\sup\{t>0 \colon \ S(\ell(s))\ \text{is calibrable for every}\ s \in 
(0,t)\}
\]
(see Figure \ref{fig:brec}(II)). By symmetry, the breaking set--up of 
every edge of $E(t_0)=S(\ell(t_0))$ is the same, and it is given by Proposition 
\ref{r:breakablepos}(i). More precisely, every edge 
$L^\eps_i(t_0)$ has cracking multiplicity $M(L^\eps_i(t_0))=5$, and 
set--up $\mathcal{C}(L^\eps_i(t_0))=\{p_i,p_{i,b},q_{i,b},q_i\}$ with
\[
\frac{\eps}{2}+\tilde{\sigma}=|p_i-p_{i,b}|=|q_i-q_{i,b}|= 
|p_j-p_{j,b}|=|q_j-q_{j,b}| 
\ \forall i,j=1,\ldots,4,
\]
where $\tilde{\sigma}$ is the calibrability threshold defined in
\eqref{f:tildesig}.
Moreover, by Remark \ref{r:symm}, using the notation of Proposition 
\ref{r:breakablepos}, we have
$L^\eps_i(t_0)=L_i^+\cup L_i^c \cup L_i^-$, and $v^+=v^-=v^c$.

Then, by Proposition \ref{p:pinning}, the evolution admits a
unique extension $E^\eps(t)$, given
by polyrectangles with 20 edges, and defined for $t\in (t_0,t_1)$,
where  $t_1$ is the first time when an edge of $E^\eps(t)$ touches 
a discontinuity line of $g_\eps$. By simmetry, the evolution
in $(t_0,t_1)$ is fully depicted by its behavior near a vertex of
$S(\ell^\eps(t_0))$ (see Figure \ref{fig:brec}(III)): pinned edges
with zero $\varphi$--curvature are generated in the normal direction 
of the edges of  $S(\ell^\eps(t_0))$ at the breaking points, 
while the edges parallel to the edges of  
$S(\ell^\eps(t_0))$ move inward. More precisely,
the edges with positive $\varphi$--curvature move inward with
constant velocity
\begin{equation}\label{f:vc}
v_c^\varepsilon=\frac{2}{\ell_0^\eps-2\eps}+\frac{\alpha+\beta}{2}+
\frac{(\alpha-\beta)\varepsilon}{\ell_0^\eps-2 \eps},
\end{equation}
while the small edges with zero $\varphi$--curvature move inward 
with velocity $v_{\pm}^\varepsilon (t)>v_c^\varepsilon$, and
reach a discontinuity line at time $t_1$ 
(see Figure \ref{fig:brec}(IV)).
Every edge $S(\ell^\eps(t_1))$ lying on a discontinuity line of 
$g_\eps$ has zero $\varphi$--curvature, and $v^{in}=\alpha$, 
$v^{out}=\beta$, and, by Proposition \ref{p:pinning}, there exists 
a unique extension of the evolution after $t_1$, given by polyrectangles
with pinned edges with zero $\varphi$--curvature, and moving edges with 
positive 
$\varphi$--curvature. If we denote by $t_2$ the first time when 
those edges reach the discontinuity lines, in such a way that
$E^\eps(t_2)$ becomes again a square $S(\ell_0^\eps-2\eps)$, then we have
\[
\eps > v_c^\varepsilon (t_2-t_1)> (k+o(\eps))(t_2-t_1),
\]
where $k$ is a constant independent of $\eps$, and hence the evolution
recomposes the square in a time lapse of order $\eps$ (see Figure 
\ref{fig:brec}(V)).
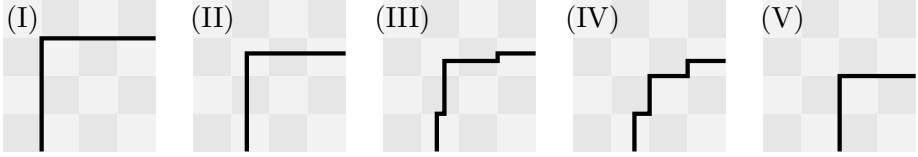
\begin{figure}[h!]

\begin{tikzpicture}[scale=1]
\foreach \x in {1,2}
\foreach \y in {1,2}
{
\draw[fill]  (\x,\y) circle [radius=0.05];
\filldraw[fill=gray!20!white,draw=none] (\x,\y) rectangle +(.5,.5);
\filldraw[fill=gray!20!white,draw=none] (\x,\y) rectangle +(-.5,-.5);
\filldraw[fill=gray!10!white,draw=none] (\x,\y) rectangle +(.5,-.5);
\filldraw[fill=gray!10!white,draw=none] (\x,\y) rectangle +(-.5,.5);
}
\draw[ultra thick] (1,0.5) -- (1,2) -- (2.5,2);
\node at (0.75,2.25) {(I)};

\foreach \x in {3.5,4.5}
\foreach \y in {1,2}
{
\draw[fill]  (\x,\y) circle [radius=0.05];
\filldraw[fill=gray!20!white,draw=none] (\x,\y) rectangle +(.5,.5);
\filldraw[fill=gray!20!white,draw=none] (\x,\y) rectangle +(-.5,-.5);
\filldraw[fill=gray!10!white,draw=none] (\x,\y) rectangle +(.5,-.5);
\filldraw[fill=gray!10!white,draw=none] (\x,\y) rectangle +(-.5,.5);
}
\draw[ultra thick] (1.2+2.5,0.5) -- (1.2+2.5,1.8) -- (2.5+2.5,1.8);
\node at (0.75+2.5,2.25) {(II)};

\foreach \x in {6,7}
\foreach \y in {1,2}
{
\draw[fill]  (\x,\y) circle [radius=0.05];
\filldraw[fill=gray!20!white,draw=none] (\x,\y) rectangle +(.5,.5);
\filldraw[fill=gray!20!white,draw=none] (\x,\y) rectangle +(-.5,-.5);
\filldraw[fill=gray!10!white,draw=none] (\x,\y) rectangle +(.5,-.5);
\filldraw[fill=gray!10!white,draw=none] (\x,\y) rectangle +(-.5,.5);
}
\draw[ultra thick] (1.2+5,0.5) -- (1.2+5,1) -- (1.3+5,1)  -- (1.3+5,1.7) --
(2+5,1.7) -- (2+5,1.8) --
(2.5+5,1.8);
\node at (0.75+5,2.25) {(III)};

\foreach \x in {8.5,9.5}
\foreach \y in {1,2}
{
\draw[fill]  (\x,\y) circle [radius=0.05];
\filldraw[fill=gray!20!white,draw=none] (\x,\y) rectangle +(.5,.5);
\filldraw[fill=gray!20!white,draw=none] (\x,\y) rectangle +(-.5,-.5);
\filldraw[fill=gray!10!white,draw=none] (\x,\y) rectangle +(.5,-.5);
\filldraw[fill=gray!10!white,draw=none] (\x,\y) rectangle +(-.5,.5);
}
\draw[ultra thick] (1.3+7.5,0.5) -- (1.3+7.5,1) -- (1.5+7.5,1)  -- 
(1.5+7.5,1.5) --
(2+7.5,1.5) -- (2+7.5,1.7) --
(2.5+7.5,1.7);
\node at (0.75+7.5,2.25) {(IV)};

\foreach \x in {11,12}
\foreach \y in {1,2}
{
\draw[fill]  (\x,\y) circle [radius=0.05];
\filldraw[fill=gray!20!white,draw=none] (\x,\y) rectangle +(.5,.5);
\filldraw[fill=gray!20!white,draw=none] (\x,\y) rectangle +(-.5,-.5);
\filldraw[fill=gray!10!white,draw=none] (\x,\y) rectangle +(.5,-.5);
\filldraw[fill=gray!10!white,draw=none] (\x,\y) rectangle +(-.5,.5);
}
\draw[ultra thick] (1.5+10,0.5) -- (1.5+10,1.5)  -- 
(2.5+10,1.5);
\node at (0.75+10,2.25) {(V)};
\end{tikzpicture}
\caption{The breaking and recomposing phenomenon} \label{fig:brec}
\end{figure} 

Since $E^\eps(t_2)$ is a square with every vertex which is also a
vertex of of a square $Q \in \acube$, $Q\subseteq S(\ell_0^\eps)$, 
the (unique) evolution then either iterates this ``breaking and recomposing'' 
motion, if $\ell^\eps(t_2)=\ell_0^\eps-2\eps>  4/(\beta-\alpha)$, or 
it is a family of 
shrinking squares, if $\ell^\eps(t_2)\leq 4/(\beta-\alpha)$.
In any case, $E^\eps(t)$ can be approximate, in the Hausdorff topology 
and locally uniformly 
in 
time, by a family of squares with side length satisfying \eqref{f:epsshrss}, so 
that the limit motion as $\eps \to 0$ is a family
of squares $S(\ell(t))$  governed by the evolution 
law \eqref{f:shrss}. 

Moreover, for every square $S(\ell_0^\eps)$ such that $
d_H(S(\ell_0),S(\ell_0^\eps))<\varepsilon$, the forced crystalline 
evolution generates and absorbs the small edges near its corners in slightly 
different 
ways, but it is always approximable by a family of squares with side length 
satisfying \eqref{f:epsshrss}. In particular, the 
limit evolution does not depend on the choice of 
the approximating data. 

\medskip

\noindent Case (ii): $\alpha+\beta<0$, and $\ell_0 > 
-4/(\alpha+\beta)$ (confinement).

\smallskip

If every vertex of $S(\ell_0^\eps)$ 
is also a
vertex of of a square $Q \in \acube$, $Q\subseteq S(\ell_0^\eps)$ 
(see Figure \ref{fig:barr}(I)), then the edges of the square lie on 
discontinuity lines of $g_\eps$, and 
they have the same velocities  $v_{0,\eps}^{in}$, 
$v_{0,\eps}^{out}<0$. 
Hence, by Definition \ref{d:crsu} and 
Proposition \ref{p:fracon}, we have 
\[
M(L^\eps)=M\left(L^\eps- \frac{\eps}{4}\nu(L^\eps)\right)=5,
\qquad \forall L^\eps \subseteq \partial S(\ell_0^\eps),
\]
and, by Proposition \ref{r:breakablepos}, every edge of $S(\ell_0^\eps)$
is splitted as  $L^\eps_i=L_i^+\cup L_i^c \cup L_i^-$ 
with $v^{in}_{\pm}=\alpha$, $v^{out}_{\pm}=\beta$, and 
$v_c^{in}, v_c^{out}<0$. Then, 
by Proposition \ref{p:pinning}, there exists a unique
forced crystalline flow starting from $S(\ell_0^\eps)$
(see Figure \ref{fig:barr}(II)), producing  small pinned corners having 
edges with zero $\varphi$--curvature and length $\eps/2$, while
he long edges with positive $\varphi$--curvature move outward with constant
velocity $v_c$ given by \eqref{f:vc} until they reach, the next discontinuity 
line (see Figure 
\ref{fig:barr}(III)). 

\begin{figure}[h!]
\begin{tikzpicture}[scale=1]
\foreach \x in {1,2,3}
\foreach \y in {1,2}
{
\draw[fill]  (\x,\y) circle [radius=0.05];
\filldraw[fill=gray!20!white,draw=none] (\x,\y) rectangle +(.5,.5);
\filldraw[fill=gray!20!white,draw=none] (\x,\y) rectangle +(-.5,-.5);
\filldraw[fill=gray!10!white,draw=none] (\x,\y) rectangle +(.5,-.5);
\filldraw[fill=gray!10!white,draw=none] (\x,\y) rectangle +(-.5,.5);
}
\draw[ultra thick] (1.5,0.5) -- (1.5,1.5)  -- 
(3.5,1.5);
\node at (0.75,2.25) {(I)};

\foreach \x in {4.5,5.5,6.5}
\foreach \y in {1,2}
{
\draw[fill]  (\x,\y) circle [radius=0.05];
\filldraw[fill=gray!20!white,draw=none] (\x,\y) rectangle +(.5,.5);
\filldraw[fill=gray!20!white,draw=none] (\x,\y) rectangle +(-.5,-.5);
\filldraw[fill=gray!10!white,draw=none] (\x,\y) rectangle +(.5,-.5);
\filldraw[fill=gray!10!white,draw=none] (\x,\y) rectangle +(-.5,.5);
}
\draw[ultra thick] (1.2+3.5,0.5) -- (1.2+3.5,1) -- (1.5+3.5,1)  -- 
(1.5+3.5,1.5) --
(2+3.5,1.5) -- (2+3.5,1.8) --
(3.5+3.5,1.8);
\node at (0.75+3.5,2.25) {(II)};

\foreach \x in {8,9,10}
\foreach \y in {1,2}
{
\draw[fill]  (\x,\y) circle [radius=0.05];
\filldraw[fill=gray!20!white,draw=none] (\x,\y) rectangle +(.5,.5);
\filldraw[fill=gray!20!white,draw=none] (\x,\y) rectangle +(-.5,-.5);
\filldraw[fill=gray!10!white,draw=none] (\x,\y) rectangle +(.5,-.5);
\filldraw[fill=gray!10!white,draw=none] (\x,\y) rectangle +(-.5,.5);
}
\draw[ultra thick] (1+7,0.5) -- (1+7,1) -- (1.5+7,1)  -- 
(1.5+7,1.5) --
(2+7,1.5) -- (2+7,2) --
(3.5+7,2);
\node at (0.75+7,2.25) {(III)};
\end{tikzpicture}
\caption{The cutting phenomenon} \label{fig:barr}
\end{figure}
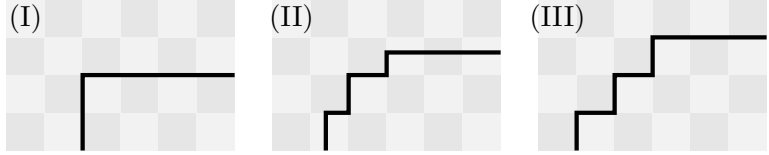

Then the process iterates, 
``cutting'' the square and reducing the length $\ell^\eps(t)$ of the edges with 
positive $\varphi$--curvature, so that their
(piecewise constant) velocity is given by
\[
v^\eps(t)= \frac{2}{\ell^\eps(t)} +\frac{\alpha+\beta}{2}- 
\frac{(\beta-\alpha)\varepsilon}{4 \ell^\eps(t)},
\] 
until the first time $t_0$ when $|v^\eps(t_0)+4/(\alpha+\beta)|<\eps$.
By Proposition \ref{p:pinning} (see also Remark \ref{r:equilibria})
every edge of $E^\eps(t)$ is pinned for $t>t_0$.
In conclusion, if we denote by $E(t)$ the family of octagons 
$E(t)=S(\ell_0)^\circ 
\cap S(\tilde{\ell}(t))$, 
where $S(\ell_0)^\circ$ is the polar square of $S(\ell_0)$, and 
$\tilde{\ell}(t)$ is the
solution to \eqref{f:shrss}, we have $d_H(E^\eps(t)-E(t))\leq c \eps$, 
so that $E^\eps(t)$ converges to $E(t)$ in the 
Hausdorff
topology and locally uniformly in time, as $\varepsilon \to 0$.

\begin{figure}[h!]
\begin{tikzpicture}[scale=1.2] 
\foreach \x in {1.2,1.3,1.38,1.45,1.5,1.53}
{
\draw[thick] (-2+\x,\x) -- (2-\x,\x) -- (\x,2-\x) -- (\x,-2+\x) -- (2-\x,-\x) 
-- (-2+\x,-\x) -- (-\x,-2+\x) -- (-\x,2-\x) -- cycle;
}
\draw[very thick] (-1,1) -- (1,1) -- (1,-1) -- (-1,-1) -- cycle;
\draw[very thick, ->] (0,1) -- (0,1.5);
\draw[very thick, ->] (0,-1) -- (0,-1.5);
\draw[very thick, ->] (1,0) -- (1.5,0);
\draw[very thick, ->] (-1,0) -- (-1.5,0);
\draw[dashed] (-2,0) -- (0,2) -- (2,0) -- (0,-2) -- cycle;
\end{tikzpicture}
\caption{The effective evolution in Case 2 of confinement.} \label{fig:pinn}
\end{figure}
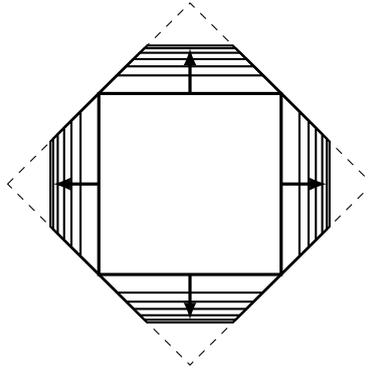
Finally, notice that the forced crystalline evolution starting from a 
general 
initial datum 
$S(\ell_0^\eps)$, with $
d_H(S(\ell_0),S(\ell_0^\eps))<\varepsilon$, reaches a configuration of the 
type depicted in Figure \ref{fig:barr}(III) in a time span of order $\eps$. 
Then the  macroscopic limit $E(t)$ does not depend on the choice of the 
approximating initial datum, and it is the effective motion of the square 
$S(\ell_0)$.  
\end{proof}

The arguments used in the proof of Theorem \ref{d:square} can be performed to 
deal with every polyrectangle (and 
hence, by approximation, to describe the effective evolution of general sets), 
but a detailed analysis of the forced crystalline flow in these 
cases requires considerable additional computation.
 Just to appreciate the 
application of the previous arguments in 
a slightly more general setting, we devote the end of this section to a 
coincise description of the motion starting from coordinate rectangles.

In what follows $R(\ell_1,\ell_2)$ will denote a coordinate rectangle with side 
lengths 
$\ell_1,\ell_2>0$. Having already characterized the evolution of a square,
without loss of generality we can assume that $\ell_{1,0}>\ell_{2,0}$.

\begin{teor}[Effective motion of coordinate rectangles]\label{d:rect}
Let $R(\ell_{1,0},\ell_{2,0})$ be a given coordinate rectangle. For every 
$\eps>0$, let 
$R(\ell_{1,0}^\eps,\ell_{2,0}^\eps)$ be a
coordinate rectangle  such that
$
d_H(R(\ell_{1,0},\ell_{2,0}),R(\ell_{1,0}^\eps,\ell_{2,0}^\eps))<\varepsilon 
$.
Then there exists a forced crystalline curvature flow 
$E^\eps(t)$, $t\in [0,T)$,
starting from $R(\ell_{1,0}^\eps,\ell_{2,0}^\eps)$. Moreover, there exists a 
family of 
sets $E(t)$, $t\in [0,T)$, such that  $E^\eps(t)$ converges to $E(t)$ in the 
Hausdorff
topology and locally uniformly in time, as $\varepsilon \to 0$.
The limit evolution $E(t)$ is independent of
the choice of the approximating initial data  
$R(\ell_{1,0}^\eps,\ell_{2,0}^\eps)$, but its 
geometry depends on the values
\[
 v_{i,0}:=\frac{2}{\ell_{i,0}^\eps} +\frac{\alpha+\beta}{2}, \qquad i=1,2 
\]
in the following way.
\begin{itemize}
\item[(i)] If $v_{i,0}\geq 0$, $i=1,2$, then $E(t)$ is a family of 
rectangles $E(t)=R(\ell_{1}(t),\ell_{2}(t))$  with
$\ell_i(t)$ governed by the system of ODEs
\begin{equation}\label{f:effrett}
\begin{cases}\displaystyle
\ell_1'= -\frac{4}{\ell_{2}} -(\alpha+\beta), \\[8pt]
\displaystyle
\ell_2'= -\frac{4}{\ell_{1}} -(\alpha+\beta), \\
\ell_1(0)= \ell_{1,0},\\
\ell_2(0)= \ell_{2,0}.
\end{cases}
\end{equation}
\item[(ii)] If either $v_{i,0}<  0$, $i=1,2$, or $v_{1,0}< 0$, 
$v_{2,0}\geq 0$ and $v_{1,0}+v_{2,0}\leq 0$, then 
$E(t)$ is a family of octagons $E(t)=Q(R(\ell_{1,0},\ell_{2,0})) \cap 
R(\tilde{\ell}_{1}(t),\tilde{\ell}_{2}(t))$, where 
$Q(R(\ell_{1,0},\ell_{2,0}))$ is the rotated square touching from outside
$R(\ell_{1,0},\ell_{2,0})$ at its vertices, and 
$(\tilde{\ell}_1,\tilde{\ell}_2)$ is the solution to \eqref{f:effrett}.
In particular, the lengths of the moving edges of $E(t)$
are governed by the system of ODEs
\begin{equation}\label{f:decsys}
\begin{cases}
\ell_1'=\dfrac{4}{\ell_1}+(\alpha+\beta), \\[6pt]
\ell_2'=\dfrac{4}{\ell_2}+(\alpha+\beta),\\
\ell_1(0)= \ell_{1,0},\\
\ell_2(0)= \ell_{2,0}.
\end{cases}
\end{equation}
\end{itemize}
\item[(iii)] If $v_{1,0}<0$, $v_{2,0}\geq 0$
and $v_{1,0}+v_{2,0}> 0$, then there exists $T>0$ such that, for $t\in [0,T]$, 
$E(t)$ is
a family of rectangles $E(t)=R(\ell_{1}(t),\ell_{2}(t))$  with
$\ell_i(t)$ governed by the system of ODEs \eqref{f:effrett}.  If 
$T=+\infty$, then $E(t)$ converges to an equilibrium (either a point or
the square $S(-4/(\alpha+\beta))$) for $t\to +\infty$. If $T<+\infty$, 
$E(t)$, $t\in [0,T)$, is the family of octagons following the rules of case 
(ii). 
\end{teor}

\begin{proof}

When either $v_{i,0}\geq 0$, or $v_{i,0}<  0$, $i=1,2$, the proof is very
similar to the one of Theorem \ref{d:square}, hence we address our attention 
to initial data with $v_{1,0}\leq 0$ and  $v_{2,0}>0$ (mixed case).

Assume that every vertex of the approximating initial datum 
$R(\ell_{1,0}^\eps,\ell_{2,0}^\eps)$
is  also 
a vertex of a square $Q \in \acube$, $Q\subseteq 
R(\ell_{1,0}^\eps,\ell_{2,0}^\eps)$. By symmetry, it is enough depict the
evolution of two contiguous edges $L^\eps_1$ (horizontal) and $L^\eps_2$ 
(vertical) starting as in Figure \ref{fig:rett3}. We have that
$v^{in}_{L^\eps_1}, v^{out}_{L^\eps_1}\leq 0$, so that
\[
M(L^\eps_1)=M\left(L^\eps_1-\frac{\eps}{4}\nu(L^\eps_1)\right)=3,
\]
while $v^{in}_{L^\eps_2}, v^{out}_{L^\eps_2}\geq 0$, and
\[
M(L^\eps_2)=M\left(L^\eps_2+\frac{\eps}{4}\nu(L^\eps_2)\right)=1.
\]
Therefore the forced evolution starts breaking the edge $L_1^\eps$, and 
generating small pinned edges with zero $\varphi$--curvature (see Figure 
\ref{fig:rett3}(II)).
The edges with 
positive curvature move with constant velocities $v_1^\eps$ (outward) and
$v_2^\eps$ (inward).

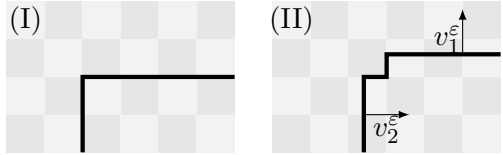
\begin{figure}[h!]
\begin{tikzpicture}[scale=1]
\foreach \x in {1,2,3}
\foreach \y in {1,2}
{
\draw[fill]  (\x,\y) circle [radius=0.05];
\filldraw[fill=gray!20!white,draw=none] (\x,\y) rectangle +(.5,.5);
\filldraw[fill=gray!20!white,draw=none] (\x,\y) rectangle +(-.5,-.5);
\filldraw[fill=gray!10!white,draw=none] (\x,\y) rectangle +(.5,-.5);
\filldraw[fill=gray!10!white,draw=none] (\x,\y) rectangle +(-.5,.5);
}
\draw[ultra thick] (1.5,0.5) -- (1.5,1.5)  -- 
(3.5,1.5);
\node at (0.75,2.25) {(I)};

\foreach \x in {4.5,5.5,6.5}
\foreach \y in {1,2}
{
\draw[fill]  (\x,\y) circle [radius=0.05];
\filldraw[fill=gray!20!white,draw=none] (\x,\y) rectangle +(.5,.5);
\filldraw[fill=gray!20!white,draw=none] (\x,\y) rectangle +(-.5,-.5);
\filldraw[fill=gray!10!white,draw=none] (\x,\y) rectangle +(.5,-.5);
\filldraw[fill=gray!10!white,draw=none] (\x,\y) rectangle +(-.5,.5);
}
\draw[ultra thick] (1.7+3.5,0.5) --  
(1.7+3.5,1.5) --
(2+3.5,1.5) -- (2+3.5,1.8) --
(3.5+3.5,1.8);
\node at (0.75+3.5,2.25) {(II)};

\draw[->] (1.7+3.5, 1) -- (2.3+3.5,1);
\node at (2+3.5,0.8) {$v_2^\eps$};

\draw[->] (3+3.5, 1.8) -- (3+3.5,2.4);
\node at (2.8+3.5,2) {$v_1^\eps$};

\end{tikzpicture}
\caption{How the mixed case starts.} \label{fig:rett3}
\end{figure}

Setting
\begin{equation}\label{f:uzero}
U(\ell_1,\ell_2):=\frac{1}{\ell_{1}}+\frac{1}{\ell_{2}} +\frac{\alpha+\beta}{2},
\end{equation}
the subsequent evolution depends on the sign of 
$U_0=U(\ell_{1,0},\ell_{1,0})$  
(note that $\alpha+\beta <0$ in this case).

\begin{figure}[h!]
\begin{tikzpicture}[scale=1.2]
\foreach \x in {4.5,5.5,6.5}
\foreach \y in {1,2}
{
\draw[fill]  (\x,\y) circle [radius=0.05];
\filldraw[fill=gray!20!white,draw=none] (\x,\y) rectangle +(.5,.5);
\filldraw[fill=gray!20!white,draw=none] (\x,\y) rectangle +(-.5,-.5);
\filldraw[fill=gray!10!white,draw=none] (\x,\y) rectangle +(.5,-.5);
\filldraw[fill=gray!10!white,draw=none] (\x,\y) rectangle +(-.5,.5);
}
\draw[ultra thick] (2+3.5,0.5) -- (2+3.5,1.8)  --
(3.5+3.5,1.8);
\node at (0.8+3.5,2.25) {(III)};
\node at (1+3.5,1.8) {$U_0\geq 0$};
\foreach \x in {1,2,3}
\foreach \y in {1,2}
{
\draw[fill]  (\x,\y) circle [radius=0.05];
\filldraw[fill=gray!20!white,draw=none] (\x,\y) rectangle +(.5,.5);
\filldraw[fill=gray!20!white,draw=none] (\x,\y) rectangle +(-.5,-.5);
\filldraw[fill=gray!10!white,draw=none] (\x,\y) rectangle +(.5,-.5);
\filldraw[fill=gray!10!white,draw=none] (\x,\y) rectangle +(-.5,.5);
}
\draw[ultra thick] (1.8,0.5) -- (1.8,1.5)  --
(2,1.5) -- (2,2) -- (2.5,2) -- (2.5,2.3) --
(3.5,2.3);
\node at (0.8,2.25) {(III)};
\node at (1,1.8) {$U_0<0$};
\end{tikzpicture}
\caption{How the mixed case carries on.} \label{fig:rett4}
\end{figure}
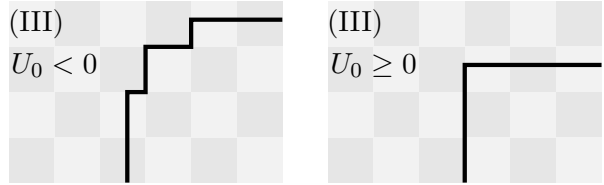

If $U_0< 0$, so that $v_{2}^\eps< -v_{1}^\eps$, the small pinned edges
with zero 
$\varphi$--curvature and ``slope 45 degrees'' are not absorbed    
(see Figure \ref{fig:rett4}, left), and the evolution ``cuts the 
vertices''. On the other hand,
the edges with positive $\varphi$--curvature move with velocities
\[
v_i^\eps(t)= \frac{2}{\ell_i^\eps(t)} +\frac{\alpha+\beta}{2}- 
\frac{(\beta-\alpha)\varepsilon}{4 \ell_i^\eps(t)}, \qquad i=1,2.
\]
Then, the effective evolution $E(t)$, in the limit $\eps\to 0$, is given by 
the family of octagons whose moving edges have lengths 
$(\ell_{1}(t),(\ell_{2}(t))$ solution to
\eqref{f:decsys}.
Since the level set
$\{U\leq 0\}$ is invariant  
under the flow of the ODEs system \eqref{f:decsys}, $E(t)$ are octagons 
(not monotonically) converging to a stationary octagon as $t \to +\infty$ (see 
Figure \ref{fig:rett6}).

If $U_0=0$, then in a time--lapse of order $\eps$ the evolution becomes a 
rectangle with the same features of the initial datum, but with $U_0^\eps<0$.
Then the effective evolution is the one depicted above.

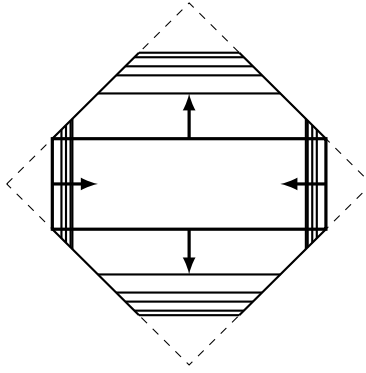
\begin{figure}[h!] 
\subfloat 
{%
\begin{tikzpicture} [scale=1.2]
\foreach \x in {1.5,1.7,1.8,1.9,1.95}
{
\draw[thick] (-2.5+\x,\x-0.5) -- (2.5-\x,\x-0.5) ;
\draw[thick] (-2.5+\x,-\x+0.5) -- (2.5-\x,-\x+0.5) ;
}
\foreach \x in {-1.4,-1.35,-1.3,-1.28 }
{
\draw[thick] (\x,\x+2) -- (\x,-\x-2) ;
}
\foreach \x in {1.4,1.35,1.3,1.28}
{
\draw[thick] (\x,-\x+2) -- (\x,\x-2) ;
}
\draw[thick] (-1.5,-1.5+2) -- (1.95-2.5,1.95-0.5);
\draw[thick] (1.5,1.5-2) -- (-1.95+2.5,-1.95+0.5);
\draw[thick] (1.5,-1.5+2) -- (-1.95+2.5,1.95-0.5);
\draw[thick] (-1.5,1.5-2) -- (1.95-2.5,-1.95+0.5);
\draw[very thick, ->] (0,0.5) -- (0,1) ;
\draw[very thick, ->] (0,-0.5) -- (0,-1);
\draw[very thick, ->] (1.5,0) -- (1,0);
\draw[very thick, ->] (-1.5,0) -- (-1,0);
\draw[very thick] (-1.5,0.5) -- (1.5,0.5) -- (1.5,-0.5) -- (-1.5,-0.5) -- cycle;
\draw[dashed] (-2,0) -- (0,2) -- (2,0) -- (0,-2) -- cycle;
\end{tikzpicture}
}
%
%
\caption{Effective evolutions, case 3 and $U_0\leq 0$}\label{fig:rett6}
\end{figure}

If $U_0> 0$, then the effective evolution maintains the rectangular shape for a 
short time
(see Figure \ref{fig:rett7}, left). Namely, the 
forced 
crystalline flow becomes a rectangle after a time lapse of order $\eps$
(see Figure \ref{fig:rett4}), and then the geometric motion $E^\eps(t)$ 
is given by  ``almost 
rectangles'', that is rectangles with small perturbations of order 
$\varepsilon$ near the vertices, until the lengths of the edges with
positive $\varphi$--curvature have velocities 
$v_{2}^\eps(t)> -v_{1}^\eps(t)$. Hence there exists $T>0$ such that
$E^\eps(t)$ it can be approximated, in the Hausdorff 
topology and locally uniformly in $[0,T]$, by a family of rectangles 
$R(\ell_1^\eps(t),\ell_2^\eps(t))$ satisfying
\[
\begin{cases}\displaystyle
(\ell_1^\eps)'= -2\left(\frac{2}{\ell_{2}^\eps} +\frac{\alpha+\beta}{2}- 
\frac{(\beta-\alpha)\varepsilon}{4 \ell_{2}^\eps}\right), \\[8pt]
\displaystyle
(\ell_2^\eps)'= -2\left(\frac{2}{\ell_{1}^\eps} +\frac{\alpha+\beta}{2}- 
\frac{(\beta-\alpha)\varepsilon}{4 \ell_{1}^\eps}\right).
\end{cases}
\]
Passing to the limit as $\eps \to 0$, we obtain  
a family of rectangles 
$E(t)=R(\ell_1(t),\ell_2(t))$, $t\in [0,T]$ where $(\ell_1,\ell_2)$ is 
the solution of the system of 
ODEs 
\eqref{f:effrett} with initial datum in the set
\[
A:=\{U(\ell_1,\ell_2)>0\}\cap
\left\{\ell_2\leq \frac{-4}{\alpha+\beta} \leq \ell_1\right\}.
\]
Notice that the function
\[
J(\ell_1,\ell_2)= 4(\log (\ell_2) - \log (\ell_1))+ 
(\alpha+\beta) (\ell_2-\ell_1)
\]
is a constant of motion for system \eqref{f:effrett}. 
The phase portrait is shown in Figure \ref{fig:rett7}. In particular,
$A$
is not a positively invariant set for the system, and the behavior of the 
trajectories depends on the energy level $J(\ell_{1,0},\ell_{2,0})$ of the 
initial datum.  
 
\begin{figure}[h!] 
\begin{tikzpicture}[scale=1.2]
\draw[thick] (-1.3,0.6) -- (1.3,0.6) -- (1.3,-0.6) -- (-1.3,-0.6) -- cycle;
\draw[thick] (-1.1,0.8) -- (1.1,0.8) -- (1.1,-0.8) -- (-1.1,-0.8) -- cycle;
\draw[dashed] (-2,0) -- (0,2) -- (2,0) -- (0,-2) -- cycle;

\draw[very thick] (-1.5,0.5) -- (1.5,0.5) -- (1.5,-0.5) -- (-1.5,-0.5) -- 
cycle;
\draw[very thick, ->] (0,0.5) -- (0,1) ;
\draw[very thick, ->] (0,-0.5) -- (0,-1);
\draw[very thick, ->] (1.5,0) -- (1,0);
\draw[very thick, ->] (-1.5,0) -- (-1,0);
\end{tikzpicture}
\hskip1.5cm\includegraphics[scale=0.55]{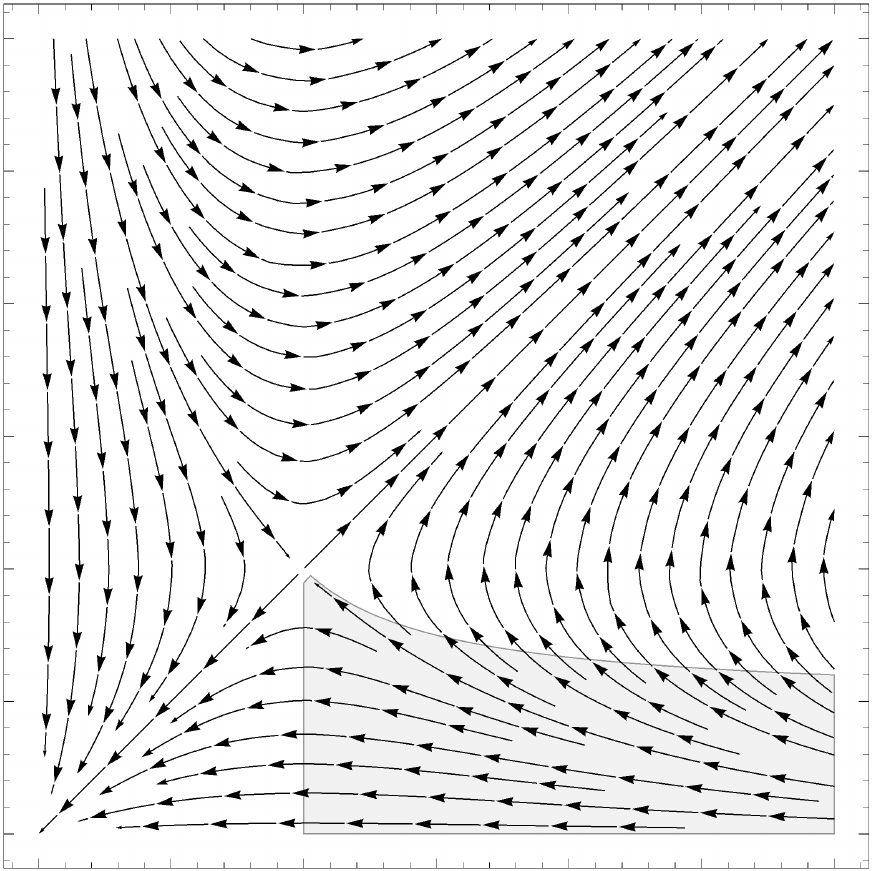} 
\caption{Left: short-time effective evolution, case 3 and $U_0> 0$.\\ Right: 
phase 
portrait of \eqref{f:effrett}, 
with the region 
$A$.}\label{fig:rett7}
\end{figure}
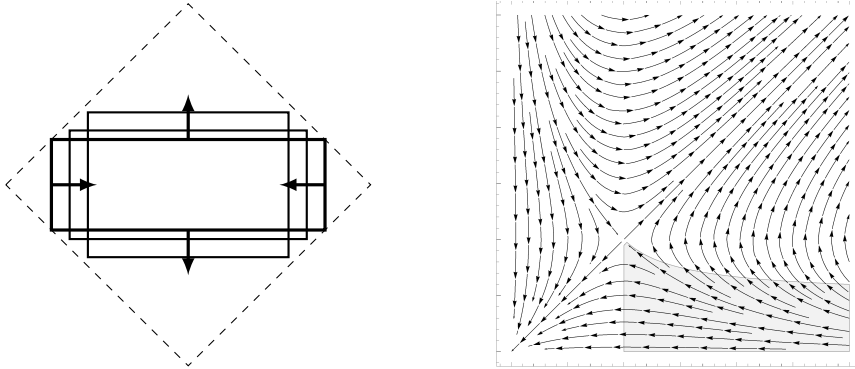

The level set $\{J=0\}$ is positively invariant in $A$, so that,
if $J(\ell_{1,0},\ell_{2,0})=0$, the effective evolution is given by rectangles 
converging,  as $t\to +\infty$, to 
the equilibrium square $S(-4/(\alpha+\beta))$.

If $J(\ell_{1,0},\ell_{2,0})<0$, then there exists a unique $t_0>0$ such 
that $\ell_1'(t_0)= -4/(\alpha+\beta)$, and the effective evolution for $t>t_0$ 
is the one shown in case (i): rectangles shrinking to a point in finite time.

If $J(\ell_{1,0},\ell_{2,0})>0$, then the solution enters in the region
$\{U \leq 0\}$ in finite time, so that the effective evolution 
becomes a family of 
octagons, converging to a stationary octagon in infinite time.
\end{proof}


\end{document}